\documentclass[12pt,a4paper,reqno]{amsart}
\usepackage[all]{xy} 
\usepackage{amsmath}
\usepackage{amsfonts}
\usepackage{amssymb}
\usepackage{mathrsfs}

\usepackage{tensor}

\usepackage{xcolor}
\usepackage{graphicx}
\newcommand{\Bmu}{\mbox{$\raisebox{-0.59ex}
  {$l$}\hspace{-0.18em}\mu\hspace{-0.88em}\raisebox{-0.98ex}{\scalebox{2}
  {$\color{white}.$}}\hspace{-0.416em}\raisebox{+0.88ex}
  {$\color{white}.$}\hspace{0.46em}$}{}}

\usepackage{amsmath,calligra,mathrsfs}

\newcommand{\scHom}{\mathscr{H}\text{\kern -3pt {\calligra\large om}}}
\newcommand{\scExt}{\mathscr{E}\text{\kern -3pt {\calligra\large xt}}}

\DeclareMathAlphabet{\mathbbold}{U}{bbold}{m}{n}

\DeclareFontEncoding{OT2}{}{} 
\newcommand{\textcyr}[1]{%
 {\fontencoding{OT2}\fontfamily{wncyr}\fontseries{m}\fontshape{n}
 \selectfont #1}}
\newcommand{\Sha}{{\!\be\lbe\mbox{\textcyr{Sh}}}}

\theoremstyle{plain}
\newtheorem{theorem}{Theorem}[subsection]
\newtheorem{mytheorem}{Theorem}

\newtheorem{lemma}[theorem]{Lemma}
\newtheorem{proposition}[mytheorem]{Proposition}
\newtheorem{myproposition}[theorem]{Proposition}
\newtheorem{definition}[theorem]{Definition}

\def\le{\kern 0.03em}
\def\A{{\mathbb A}}

\def\F{{\mathbb F}}

\def\O{{\mathcal O}}

\def\Q{{\mathbb Q}}

\def\U{{\mathcal U}}
\def\V{{\mathcal V}}
\def\X{{\mathcal X}}

\def\Z{{\mathbb Z}}

\def\e{\kern 0.08em}
\def\be{\kern -.1em}
\def\lbe{\kern -.025em}

\DeclareMathOperator{\coh}{H}

\DeclareMathOperator{\Gal}{Gal}

 \DeclareMathOperator{\Nm}{N}
 \DeclareMathOperator{\image}{Im}
\DeclareMathOperator{\Sel}{{\rm{Sel}}}

\DeclareMathOperator{\Hom}{Hom} 
 
\DeclareMathOperator{\ord}{ord} \DeclareMathOperator{\coker}{coker}

\begin{document}
\title{The Frobenius twists of elliptic curves over global function fields}
\author{Ki-Seng Tan}
\address{Department of Mathematics\\
National Taiwan University\\%
Taipei 10764, Taiwan}
\email{tan@math.ntu.edu.tw}
\thanks{\textbf{Acknowledgement:} This research was supported in part by Ministry of Science and Technology
of Taiwan, MOST 109-2115-M-002-008-MY2. The author thanks F. Trihan for many valuable suggestions especially for
helping him with the proof of Lemma \ref{l:isogen}}
\begin{abstract}
For an elliptic curve $A$ defined over a global function field $K$ of characteristic $p>0$, the $p$-Selmer group of the Frobenius twist $A^{(p)}$ of $A$ tends to have
larger order than that of $A$. The aim of this note is to discuss this phenomenon.

\end{abstract}
\maketitle

\section{Introduction}\label{s:int}
For an elliptic curve $A$ defined over a global function field $K$ of characteristic $p>0$, the $p$-Selmer group of the Frobenius twist of $A$ tends to have
larger order than that of $A$. The aim of this note is to discuss this phenomenon.

The Frobenius twist $A^{(p)}$ is the base change $A\times_K K$ of $A$ over the absolute Frobenius $\mathrm{Frob}_{p}:K\longrightarrow K$, $x\mapsto x^{p}$. 

\subsection{The main results}\label{su:main} We assume that $A{/K}$ is {\em{ordinary}},  
having {\em{semi-stable reduction everywhere}}.
Let $\Delta_{A/K}$ denote the divisor of the global minimal discriminant of $A/K$.
Comparing the defining equation for both curves yields
\begin{equation}\label{e:del}
\Delta_{A^{(p)}/K}=p\cdot \Delta_{A/K}.
\end{equation}

Let $A_{p^\nu}$ be the the kernel of the multiplication by $p^\nu$ on $A$ and let
$$\Sel_{p^\nu}(A/K)\subset \coh^1(K,A_{p^\nu}) $$ 
denote the $p^\nu$-Selmer group of $A/K$. 
If $p=2$, let $\eth'$ be the set of places of $K$ at which
$A$ has non-split multiplicative reduction and has the group of components of even order; otherwise, put $\eth'=\emptyset$.
Let $S_b$ denote the set of bad reduction places of $A/K$.
Write
$$S_b=\eth'\sqcup \eth.$$

Let $k=K(A^{(p)}_p(\bar K^s))$ and let $\eth_0\subset \eth$ be the subset of places splitting completely over $k$.
Let $\hslash$ denote the $p$-rank of the subgroup of $\Hom(\Gal(\bar k^s/k),\Z/p\Z)$ consisting of homomorphisms unramified everywhere 
and locally trivial at every places of $k$ sitting over $\eth$.
Our main results are as follow. 
Let $q$ denote the order of the constant field of $K$. 

\begin{proposition}\label{p:sel} There exists an integer $\epsilon_1$, $2\hslash+1+|\eth_0|\geq\epsilon_1\geq -|\eth_0|$, such that
$$  \log_p|\Sel_p(A^{(p)}/K)|=\frac{(p-1)\deg\Delta_{A/K}}{12}\cdot \log_p q+\epsilon_1.$$
\end{proposition}

 \begin{proposition}\label{p:aap} There exists an integer $\epsilon_2$, $2\hslash+1+|\eth_0|\geq\epsilon_2\geq  -2\hslash-3|\eth_0|$,
 such that for each positive
 integer $\nu$,
$$ \log_p|\Sel_{p^{\nu+1}}(A^{(p)}/K)|=\frac{(p-1)\deg\Delta_{A/K}}{12}\cdot \log_p q+\log_p|\Sel_{p^{\nu}}(A/K)|+\epsilon_2.$$
\end{proposition}
Note that $\frac{\deg\Delta_{A/K}}{12}$ is a non-negative integer (see \cite[\S 2.2.1]{lltt16}), it is zero, if and only if $A/K$ is isotrivial.
Let $\Sha_{p^\infty}(A/K)$ denote the $p$-primary part of the Tate-Shafarevich group of $A/K$, let $\Sha_{div}(A/K)$ be its $p$-divisible
subgroup, and denote the $p$-cotorsion
$$\overline\Sha(A/K):=\Sha_{p^\infty}(A/K)/\Sha_{div}(A/K)=\Sel_{p^\infty}(A/K)/\Sel_{div}(A/K),$$
where $\Sel_{div}(A/K)$ is the $p$-divisible subgroup of $\Sel_{p^\infty}(A/K)$. Let $r$ denote the $\Z_p$-co-rank of $\Sel_{div}(A/K)$.
If $\nu$ is greater than the exponents of
$\overline\Sha(A/K)$ and  $\overline\Sha(A^{(p)}/K)$, then 
$$|\Sel_{p^{\nu}}(A/K)|=|\overline\Sha(A/K)|\cdot p^{r\nu}\;\;\text{and}\;\;
|\Sel_{p^{\nu+1}}(A^{(p)}/K)|=|\overline\Sha(A^{(p)}/K)|\cdot p^{r(\nu+1)}.$$
It follows from Proposition \ref{p:aap} that
$$ \log_p|\overline\Sha(A^{(p)}/K) |+r=\frac{(p-1)\deg\Delta_{A/K}}{12}\cdot \log_p q+\log_p| \overline\Sha(A/K)|+\epsilon_2.$$
Such kind of formulae is suggested by the conjectured Birch and Swinnerton-Dyer formulae (see \cite{Ta66,tan95}) for both
$A^{(p)}/K$ and $A/K$.

Next, let $L/K$ be a $\Z_p^d$-extension ramified only at a {\em {finite number of ordinary places}}
of $A/K$. Write $\Gamma:=\Gal(L/K)$ and $\Lambda_\Gamma:=\Z_p[[\Gamma]]$. Let $Z$ be a finitely generated {\em{torsion}} 
$\Lambda_\Gamma$-module, so that there is an exact sequence 
\begin{equation}\label{e:iwasawa}
\xymatrix{0 \ar[r] &  \bigoplus_{i=1}^m \Lambda_\Gamma/(p^{\alpha_i}) \oplus \bigoplus_{j=1}^n \Lambda_\Gamma/(\eta_j^{\beta_j})
\ar[r]  &  Z \ar[r] &  N \ar[r] & 0,}
\end{equation}
where $\alpha_1,...,\alpha_m,\beta_1,...,\beta_n$ are positive integers,
$\eta_1,...,\eta_n\in \Lambda_\Gamma$ are irreducible, relatively prime to $p$, and $N$ is pseudo-null.
Although the exact sequence is not canonical, the modules  $\bigoplus_{i=1}^m \Lambda_\Gamma/(p^{\alpha_i})$ and
$\bigoplus_{j=1}^n \Lambda_\Gamma/(\eta_j^{\beta_j})$ are uniquely determined by $Z$,
we call them the $p$ part and the non-$p$ part of $Z$, call $p^{\alpha_1}, ..., p^{\alpha_m}$ the elementary $\mu$-invariants and $m$ the $\mu$-rank of $Z$.
If $Z$ is non-torsion, define the $\mu$-rank to be $\infty$.

Consider the Pontryagin dual $X$, $X^{(p)}$ of $\Sel_{p^\infty}(A/L)$, $\Sel_{p^\infty}(A^{(p)}/L)$. They are finitely generated over $\Lambda_\Gamma$ (see \cite{tan13}).
 Put
$$\eth_1:=\{ v\in \eth\;\mid \; v \;\text{splits completely over}\; kL\}.$$

\begin{proposition}\label{p:p} 
The $\mu$-rank of $X^{(p)}$ is at least
$\frac{(p-1)\deg\Delta_{A/K}}{12}\cdot \log_p q -|\eth_1|$.
\end{proposition} 
If $L$ contains $K_p^{(\infty)}$, the constant
$\Z_p$-extension over $K$, then $X$ and $X^{(p)}$ are torsion \cite{ot09,tan13}, in this case $|\eth_1|=0$.

\begin{proposition}\label{p:constant}
If $L$ contains $K_p^{(\infty)}$, then the $\mu$-rank $m$ of $X^{(p)}$ equals $ \frac{ (p-1)\deg\Delta_{A/K}}{12}\cdot \log_p q$.
Moreover, if $p^{\alpha_1}, ..., p^{\alpha_m}$, $\alpha_1\geq\cdots\geq\alpha_m$, are the elementary $\mu$-invariants of $X^{(p)}$, then those of $X$
are $p^{\alpha_1-1}, ..., p^{\alpha_{m'}-1}$, where $m'$ is the greatest $i$ such that $\alpha_i>1$.
\end{proposition}

For a finite extension $F/K$, let $\mathfrak w_F$ denote the $p$-completion of the divisor class group of of $kF$ and for a $\Z_p^e$ sub-extension $M/K$ of $L/K$, put $\mathfrak w_M:=\varprojlim_{K\subset F\subset M}\mathfrak w_F$, which is finitely generated torsion over
$\Lambda_{\Gal(kM/k)}$. Actually, by \cite{crw87}, the characteristic ideal of $\mathfrak w_M$ has a generator
$\Theta_M:=\varprojlim_F\Theta_F$, where basically for each $F$, $\Theta_F\in\Z_p[\Gal(kF/k)]$ is the Stickelberger element defined in \cite[\S V.1.1]{Ta84},
in particular, we have
\begin{equation}\label{e:stick}
p_{L/M}(\Theta_L)=\Theta_M\cdot *,
\end{equation}
where $p_{L/M}: \Lambda_{\Gal(Lk/k)}\longrightarrow  \Lambda_{\Gal(Mk/k)}$ is the continuous $\Z_p$-algebra homomorphism extending the quotient map $\Gal(Lk/k)\longrightarrow \Gal(Mk/k)$ and  $*\in \Lambda_{\Gal(Mk/k)}$ is a fudge factor not divisible by $p$. 

For simplicity, 
we shall identify 
$\Lambda_{\Gal(Mk/k)}$ with 
$\Lambda_{\Gal(M/K)}$, and view $\Theta_L$, $\mathfrak w_L$ as objects over $\Lambda_\Gamma$. 
In the special case where $L=K_p^{(\infty)}$, the module $\mathfrak w_L$ has trivial $\mu$-rank, hence $\Theta_L$ is not divisible
by $p$. To see this, let $\mathcal Y$ be the complete smooth curve defined over the constant field of $k$,
having $k$ as its function field. For every finite sub-extensions $F/K$ of $L/K$,
we have the exact sequence
$$\xymatrix{0\ar[r] & \mathfrak w_F^0 \ar[r] & \mathfrak w_F \ar^-\deg[r] & \Z_p \ar[r] & 0}$$
and $\mathfrak w_F^0[p]$
is contained in the subgroup of $p$-division points of the Jacobian variety of $\mathcal Y$.
Therefore, the order of $\mathfrak w_F^0[p]$ is bounded, and hence $\mathfrak w_L[p]=0$.
In general, \eqref{e:stick} says that if $L$ contains $K_p^{(\infty)}$, then $\Theta_L$ is not divisible by $p$. Also,
in this case, $\eth_1=\emptyset$. Hence Proposition
\ref{p:constant} is a special case of the following proposition. 

\begin{proposition}\label{p:gen}
If $\Theta_L$ is not divisible by $p$ and $\eth_1=\emptyset$, then the $\mu$-rank $m$ of $X^{(p)}$ equals $ \frac{ (p-1)\deg\Delta_{A/K}}{12}\cdot \log_p q$.
Moreover, if $p^{\alpha_1}, ..., p^{\alpha_m}$, $\alpha_1\geq\cdots\geq\alpha_m$, are the elementary $\mu$-invariants of $X^{(p)}$, then those of $X$
are $p^{\alpha_1-1}, ..., p^{\alpha_{m'}-1}$, where $m'$ is the greatest $i$ such that $\alpha_i>1$.

\end{proposition}

Since the Frobenius and Verschiebung induce pseudo isomorphisms
between the non-$p$ parts of $X$ and $X^{(p)}$, the proposition implies the characteristic ideal of $X^{(p)}$ is the $q^{\frac{(p-1)\deg\Delta_{A/K}}{12}}$ multiple of that of $X$.
If $L=K_p^{(\infty)}$, this is also a consequence of the main theorem of \cite{lltt16}.

\subsection{Notation}\label{su:not}
For a field $F$, let $\bar F$ and $\bar F^s$ denote its algebraic closure and separable closure, and denote 
$G_F=\Gal(\bar F^s/F)$.
For a place $v$, let $\O_v$, $\pi_v$ and $\F_v$ denote the ring of integers, an uniformizer and the residue field.
Write $q_v$ for $|\F_v|$.

Let $S_{ss}$ denote the set of place $v$ of $K$ at which $A$ has supersingular reduction. 
For a set $S$ of places of $K$ and an algebraic extension $F$, let $S(F)$ denote the set of places of $F$ sitting over $S$. 
For an endomorphism $\varphi$ of an abelian group $H$, let $H[\varphi]$ denote the kernel. 
Use $^\vee$ for Pontryagin dual, $\sim$ for pseudo isomorphism. 

In this note we use  flat or Galois cohomology.
Let
$$\mathsf F:A\longrightarrow A^{(p)}\;\; \text{and}\;\; \mathsf V:A^{(p)}\longrightarrow A$$ 
be the Frobenius and the Verschiebung homomorphisms. We have the exact sequences

\begin{equation}\label{e:fec}
\xymatrix{0 \ar[r]  &  C_p  \ar[r] & A_p \ar[r]^-{\mathsf{F}} &E_p^{(p)} \ar[r] & 0,}
\end{equation}
as well as
\begin{equation}\label{e:ve}
\xymatrix{0 \ar[r]  & E_p^{(p)}  \ar[r] & A_p^{(p)} \ar[r]^-{\mathsf{V}} & C_p \ar[r] & 0,}
\end{equation}
where $C_{p}=\ker \mathsf{F}$ is connected and $E_p^{(p)}=\ker \mathsf{V}$, $\acute{\text{e}}$tale (see \cite{lsc10}). 
For a field $F$ containing $K$, we have $A_p^{(p)}(F)=E_p^{(p)}(F)$. In particular, $k=K(E_p^{(p)}(\bar K^s))$.
Note that for $p=2$, $k=K$,
because the non-trivial point of $E_p^{(p)}(\bar K^s)$ is the only Galois conjugate of itself.

\subsection{The proofs}\label{su:pf}
The key ingredient in the proof of Proposition \ref{p:sel}  is local, concerning the kernel of the natural map 
$\coh^1(K_v,E_p^{(p)})\longrightarrow \coh^1(K_v,A^{(p)})$,
especially when $v$ is a place of supersingular reduction. 
Lemma \ref{l:mult}, Lemma \ref{l:good} and Lemma \ref{l:sscond} treat different types of reduction and provide us
criteria, in terms of the the conductor of the corresponding cyclic extension over $k_v$, for determining if an element 
of $\coh^1(K_v,E_p^{(p)})$ is in such kernel. 

With the local criteria available and with the help of global class field, in Proposition \ref{p:qvnv}, we determine the order of 
$\Sel(E_p^{(p)}/K)$, the $E_p^{(p)}$-part of $\Sel_p(A^{(p)}/K)$. 
In doing so, we find an interesting phenomena
that the discrepancy between the two discriminants
as described in \eqref{e:del} is solely contributed by supersingular places (see \eqref{e:discqvnv}). 
Next, in Proposition \ref{p:selcp} we determine  the order of 
$\Sel(C_p/K)$, the $C_p$-part of $\Sel_p(A/K)$, by using the Poitou-Tate duality \cite{ces15}.
Then Proposition \ref{p:sel} is proved at the end of \S\ref{su:gh1}, as a consequence of the above two propositions. 
Proposition \ref{p:aap}  is proved in \S \ref{su:p2} by applying Cassels-Tate duality.

In \S\ref{su:spec}, using a theorem of Monsky \cite{monsky}, we establish a method of reducing the proofs of Proposition \ref{p:p},
\ref{p:gen} to the $d=1$ case.
Since the method can be applied to more general situations, we loosen the condition in that subsection by allowing $A$ to be an ordinary
abelian variety defined over a global field $K$.  The main result is summarized in Lemma \ref{l:special}. 
As a consequence of this lemma and Proposition \ref{p:sel}, \ref{p:aap}, the proofs of Proposition \ref{p:p}, \ref{p:gen} are given in the final subsection \S\ref{su:last}.


\section{Local fields}\label{s:loc}
At each place $v$ of $K$, we have the long exact sequence
\begin{equation}\label{e:jv}
\xymatrix{\cdots\ar[r] &A^{(p)}(K_v) \ar[r]^-{\mathsf V_* } & A(K_v) \ar[r]^-{\partial} & \coh^1(K_v,E_p^{(p)}) \ar[r]^-{j_v} & \coh^1(K_v,A^{(p)})\ar[r] &\cdots}
\end{equation}
derived from $\xymatrix{0\ar[r] & E_p^{(p)}\ar[r]^j & A^{(p)} \ar[r]^-{\mathsf V} & A\ar[r] & 0}$. 
The aim of this section is to determine $\ker(j_v)=\coker(\mathsf V_*)$.
For a place $w$ of $k$ sitting over $v$, the abelian group
\begin{equation}\label{e:h1hom}
\coh^1(k_w,E_p^{(p)})=\Hom (G_{k_w}, \Z/p\Z)=\Hom(k_w^*/(k_w^*)^p,\Z/p\Z),
\end{equation}
so each $\xi\in\coh^1(k_w,E_p^{(p)})$ determines a degree $p$ cyclic extension $k_{w,\xi}/k_w$.

Let $\ord_v$ denote the valuation on $\bar K_v$ having $\ord_v(\pi_v)=1$.

\subsection{Frobenius and Verschiebung}\label{su:fv}

Let $\mathscr F$ (resp. $\mathscr F^{(p)}$) denote the formal group law associated to $A$ (resp. $A^{(p)}$)
over $K_v$. Since $A$ has semi-stable reduction, $\mathscr F$ is stable under local field extensions.
Let $\bar A$ denote the reduction of $A$ at $v$. 
For a place $w$ of an algebraic extension $F$ of $K$, sitting over $v$, the pre-image
$A_{o}(F_w)$ of $0\in\bar A(\F_w)$ under the reduction map $A(F_w)\longrightarrow \bar A(\F_w)$ is identified with $\mathscr F(\mathfrak m_w)$ via a bijection $\iota$. Let $\iota^{(p)}$ be the corresponding bijection associated to $A_o^{(p)}(F_w)$.

Let $\mathscr P\in \O_v[[t]]$ be the unique power series fitting into the commutative diagram
\begin{equation}\label{e:vf}
\xymatrix{\mathscr{F}(\mathfrak m_w) \ar[r]^-{\mathrm{Frob}_{p}} \ar[d]^-{\iota}& \mathscr{F}^{(p)}(\mathfrak m_w)
\ar[r]^-{\mathscr P}\ar[d]^-{\iota^{(p)}}& 
\mathscr{F}(\mathfrak m_w)\ar[d]^-{\iota}\\
A_{o}(F_w) \ar[r]^-{\mathsf F}  & A_{o}^{(p)}(F_w)\ar[r]^-{\mathsf V}& A_{o}(F_w) .}
\end{equation}

An element $\xi\in\mathfrak m_w$ satisfies $\mathscr P(\xi)=0$
if and only if $\iota^{(p)}(\xi)\in E_{p}^{(p)}(F_w)\cap A_{o}^{(p)}(F_w)$. 

Suppose the formal group law of $\mathscr F$ (resp. $\mathscr F^{(p)}$) is given by the $\O_v$-coefficient power series $f(X,Y)$
(resp. $f^{(p)}(X,Y)$). 

If $t_i$ is a solution to $\mathscr{P}(t)=0$, then $\mathscr{P}(f^{(p)}(t,t_i))=\mathscr{P}(t)$, furthermore, since $A/K_v$ is ordinary, by \cite[\S IV.7.2]{sil86}, $\mathscr{P}'(0)\not=0$, and hence $\mathscr{P}'(t_i)\not=0$, too. 

\begin{lemma}\label{l:ss} We have $\mathscr P(t)=u(t)\cdot P(t)$, where $u(t)$ is a unit in $\O_v[[t]]$ and $P(t)$ is the associated distinguished polynomial.  The polynomial $P(t)$ is separable with $P(0)=0$.
If $v$ is a supersingular place, then $\deg P=p$; if $v$ is ordinary, then $P(t)=t$.
\end{lemma} 
\begin{proof}  The first assertion follows from the claim that $\pi_v$ does not divides $\mathscr P(t)$. For ordinary $v$, because
$E_p^{(p)}(\bar K_v)\cap A_{o}^{(p)}(\bar K_v)=\{0\}$, we know that $0$ is the only root of
$P(t)$ in $\bar K_v$. This implies that the distinguished polynomial $P(t)=t$.
For supersingular $v$, the group $E_p^{(p)}(\bar{K}_v)\cap A_{o}^{(p)}(\bar{K}_v)\simeq \Z/p\Z$, so $\deg P(t)=p$, furthermore, 
since $\mathscr P(t_i)=0$
and $\mathscr{P}'(t_i)\not=0$, we have $P(t_i)=0$ and $P'(t_i)\not=0$.

To prove the claim, we first consider the case where $v$ is a place of good reduction. The formal group law associated to $\bar A$ is given by $\bar{f}(X,Y):=f(X,Y) \pmod{(\pi_v)}$, which has height $1$ or $2$, 
so $\bar{\mathscr P}:=\mathscr P\pmod{(\pi_v)}$ is non-zero. This proves the claim.

For a multiplicative place $v$, we prove the claim by showing that the Verschiebung  
gives rise to an isomorphism $\mathscr F^{(p)}(\mathfrak m_v)\longrightarrow \mathscr F(\mathfrak m_v)$.
If $v$ is a split-multiplicative place and $\tilde Q$ is the local Tate period of $A$, then $\tilde Q^p$ is the local Tate period of 
$A^{(p)}$ and the Verschiebung is given by $K_v^*/\tilde Q^{p\Z}\longrightarrow K_v^*/\tilde Q^\Z$, induced from the identity map on $K_v^*$.
This implies $\mathscr F^{(p)}(\mathfrak m_v)\longrightarrow \mathscr F(\mathfrak m_v)$ is an isomorphism, and hence the claim follows. 

If $v$ is non-split multiplicative,
then $A/K_v$ is the twist of a split multiplicative elliptic curve $B/K_v$ by the unramified quadratic extension $L_w/K_v$. 
Write $\mathsf Z_v$ for the kernel of the norm map $\Nm_{L_w/K_v}:\O_w^*\longrightarrow \O_v^*$. Then $\mathscr F^{(p)}(\mathfrak m_v)\longrightarrow \mathscr F(\mathfrak m_v)$ is given by the identity map $\mathsf Z_v\longrightarrow \mathsf Z_v$, so it is an isomorphism.

\end{proof}

\subsection{Ordinary places}\label{su:mulp} 

The proof of Lemma \ref{l:ss} shows that if $v$ is a split 
multiplicative place, then $\mathsf V_*$ is given by $K_v^*/\tilde Q^{p\Z}\longrightarrow K_v^*/\tilde Q^\Z$, and hence surjective, so by \eqref{e:jv}, $\ker (j_v)=0$.

\begin{lemma}\label{l:mult} Let $v$ be a multiplicative place. Then $\ker(j_v)=0$, unless $v\in\eth'$,
in which case $\ker(j_v)$ is of order $2=p$, 
consisting of $\xi\in\coh^1(K_v,E_p^{(p)})$ with $K_{v,\xi}/K_v$ unramified.
\end{lemma}
Note that if $p=2$, then $k=K$, so $K_{v,\xi}$ is defined. 

\begin{proof} 

Suppose $v$ is non-split multiplicative and let $B/K_v$ and $L_w/K_v$ be as in the proof of Lemma \ref{l:ss}.
Because $A/L_w$ is split-multiplicative, we have the injection $\coh^1(L_w,E_p^{(p)})\longrightarrow \coh^1(L_w,A^{(p)})$, and hence
$$\ker(j_v)=\ker (\coh^1(L_w/K_v,E_p^{(p)}(L_w))\longrightarrow \coh^1(L_w/K_v,A^{(p)}(L_w))).$$
For $p\not=2$, we have $\coh^1(L_w/K_v,E_p^{(p)}(L_w))=0$, so $\ker(j_v)=0$.  Denoting 
$G=\Gal(L_w/K_v)$, we have the commutative diagram
$$\xymatrix{\coh^1(L_w/K_v,E_p^{(p)}(L_w)) \ar[r] \ar[d]^-\simeq & \coh^1(L_w/K_v, A^{(p)}(L_w)) \ar[d]^-\simeq \\
\Hom(G,\Z/2\Z) \ar[r] & B^{(p)}(K_v)/\Nm_{L_w/K_v}(B^{(p)}(L_w)).}
$$
The non-trivial element of $\Hom(G,\Z/2\Z)$, sending the generator of $G$ to
the point of $B^{(p)}(K_v)$ obtained by the Tate local period $Q_v$ of $B/K_v$,  corresponds to an element of $\ker(j_v)$ if and only if  
$Q_v\in \Nm_{L_w/K_v}(L_w^*)$, or equivalently $\ord_v Q_v$ is even.

\end{proof}

\begin{lemma}\label{l:good}
Suppose $v$ is a good ordinary place and $w$ is a place of $k$ sitting over $v$. Then $\ker(j_w)$ is of order $p$, consisting of 
$\xi\in\coh^1(k_w,E_p^{(p)})$ with $k_{w,\xi}/k_w$ unramified. If $K_v\not=k_w$, then $\ker(j_v)$ is trivial.
\end{lemma} 
\begin{proof} In view of the diagram \eqref{e:vf}, Lemma \ref{l:ss} says $\xymatrix{A_o^{(p)}(k_w) \ar[r]_-\sim^-{\mathsf{V}} & A_o(k_w)}$.
We have to determine the cokernel of the induced $\bar{\mathsf{V}}:\bar A^{(p)}(\F_w)\longrightarrow \bar A(\F_w)$.
The Frobenius $\bar{\mathsf{F}}$ identifies $\bar A^{(p)}(\F_w)$ with $\bar A(\F_w)$ and under this, $\bar{\mathsf{V}}$ is identified with 
the multiplication by $p$. The cokernel in question is isomorphic to 
$$\bar A(\F_w)/p\bar A(\F_w)\simeq \bar A_p(\F_w)=\bar E_p^{(p)}(\F_w).$$
The snake lemma applied to the diagram
$$\xymatrix{0\ar[r] & A_o^{(p)}(k_w) \ar[r] \ar[d]^-{\mathsf{V}}_-\simeq & A^{(p)}(k_w) \ar[r] \ar[d]^-{\mathsf{V}} & \bar {A}^{(p)}(\F_w) \ar[r] 
\ar[d]^-{\bar{\mathsf{V}}} & 0\\
0 \ar[r] & A_o(k_w) \ar[r] & A(k_w) \ar[r] & \bar{A}(\F_w) \ar[r] & 0}$$
implies that the reduction map $E_p^{(p)}(k_w)\longrightarrow \bar E_p^{(p)}(\F_w)$ is an isomorphism,
so $\ker (j_w)$ is of order $p$, and by
\cite[\S I.3.8]{mil86}, it is formed by all unramified $\xi$. If $K_v\not=k_w$,
then $\bar E_p^{(p)}(\F_v)=E_p^{(p)}(K_v)=0$, and a similar argument shows $\ker(j_v)$ is trivial.

\end{proof}

\subsection{Supersingular places}\label{su:ss} 

Suppose $v$ is supersingular. Choose a nonzero $t_1\in\mathfrak m_w$ (in some $F_w$) with 
$\iota^{(p)}(t_1)\in E_p^{(p)}(F_w)$. Let $[u]$ denote the multiplication by $u$ on  
$A^{(p)}$. Because
$t_u:={\iota^{(p)}}^{-1}\circ [u]\circ \iota^{(p)} (t_1)=ut_1+\text{higher terms},$ 
if  $p\nmid u$, then $\ord_v t_u=\ord_v t_1$.
Denote
\begin{equation}\label{e:tk}
n_v:=\sum_{i\in\F_p^*} \ord_v t_i=(p-1)\ord_v(t_u), \;\text{for}\; (u,p)=1.
\end{equation} 

Write
$$P(t)=t^p+z_{p-1}t^{p-1}+\cdot+z_1t,$$
with
\begin{equation}\label{e:ord}
\ord_v z_1=n_v.
\end{equation}

For $s,t \in\mathfrak m$ write $s \boxplus t$ for $\mathscr F^{(p)}(s,t)$. Then $\iota^{(p)}(s\boxplus t)=\iota^{(p)}(s)+\iota^{(p)}(t)$.
Diagram \eqref{e:vf} shows that for a given $b\in \mathfrak m_v$,
if $a_0\in \mathfrak m_w$ is a root of $\mathscr P(t)-b=0$, then all other roots equal $a_u:=a_0\boxplus t_u=\mathscr F^{(p)}(a_0,t_u)$, $u=1,...,p-1$.
Let $Q(t)$ be the distinguished polynomial associated to $\mathscr P(t)-b$, whose roots are also $a_0,...,a_{p-1}$. Since $Q(0)=-b\cdot \xi$, for some $\xi\in \O_v^*$, 
\begin{equation}\label{e:ab}
\sum_{u=0}^{p-1}\ord_v  a_u=\ord_v b.
\end{equation}
Since
$\mathscr F^{(p)}(X,0)=\mathscr F^{(p)}(0,X)=X$, we can write
$\mathscr F^{(p)}(X,Y)=X+Y+XY\cdot g(X,Y)$.
It follows that $a_u=a_0+t_u+\text{higher terms}.$ Hence
\begin{equation}\label{e:diff}
\ord_v (a_u-a_{u'})= \frac{n_v}{p-1}.
\end{equation}

\begin{lemma}\label{l:ab} For every $b \in\mathfrak m_v$ with $\ord_v b>\frac{pn_v}{p-1}$, there exists an element $a\in \mathfrak m_v$,
with $\ord_v a>\frac{n_v}{p-1}$, such that $\mathscr P(a)=b$. Conversely, for $a\in \mathfrak m_v$,
with $\ord_v a>\frac{n_v}{p-1}$, the element $b=\mathscr P(a)\in \mathfrak m_v$ has 
$\ord_v b=\ord_v a+n_v>\frac{pn_v}{p-1}$.
\end{lemma}
\begin{proof} 
If $b \in\mathfrak m_v$ and $a_0$ is a solution to $\mathscr P(t)=b$, with $\ord_v(a_0)>\frac{n_v}{p-1}$, then by \eqref{e:diff}, other solutions $a_u$ have $\ord_v a_u=\frac{n_v}{p-1}$. Therefore, if $a\in \mathfrak m_v$,
with $\ord_v a>\frac{n_v}{p-1}$, and $b=\mathscr P(a)$, then by \eqref{e:ab}, $\ord_v b=\ord_v a+n_v$. Conversely, if 
$b \in\mathfrak m_v$, $\ord_v b>\frac{pn_v}{p-1}$, by \eqref{e:ab}, there is a solution $a$ to
$\mathscr P(t)=b$, such that $\ord_v a>\frac{n_v}{p-1}$. Comparing the valuations, we deduce that $a$ is the only Galois conjugate of itself, whence $a\in\mathfrak m_v$.

\end{proof}

\begin{lemma}\label{l:image} If $v$ is a supersingular place of $A/K$, then the cokernel of 
$$\xymatrix{A^{(p)}(K_v) \ar[r]^-{\mathsf V_{*} } & A(K_v)}$$ 
is of order $p^{\epsilon_v}\cdot q_v^{n_v}$, where if $k_v=K_v$, $\epsilon_v=1$; otherwise, $\epsilon=0$.
\end{lemma}
\begin{proof}
Since $\bar{A}(\F_v)$ has order prime to $p$, by Diagram \eqref{e:vf} we need to show the cokernel
of $\xymatrix{\mathscr{F}^{(p)}(\mathfrak m_v) \ar[r]^-{\mathscr P} & \mathscr{F}(\mathfrak m_v)}$ has the desired order.

Let $\beta=[\frac{1}{p-1}n_v]+1$. Lemma \ref{l:ab} implies that $\mathscr P$ sends $\mathscr{F}^{(p)}(\mathfrak m^{\beta})$
onto $\mathscr{F}(\mathfrak m^{\beta+n_v})$. Therefore, it is sufficient to check the co-kernel of the induced homomorphism
$$\bar{\mathscr{P}}:\mathscr{F}^{(p)}(\mathfrak m)/ \mathscr{F}^{(p)}(\mathfrak m^{\beta})
\longrightarrow \mathscr{F}(\mathfrak m)/ \mathscr{F}(\mathfrak m^{\beta+n_v}).$$
Since the kernel of $\bar{\mathscr{P}}$ is of order $p^{\epsilon_v}$, the proof is completed by counting.
\end{proof}

Lemma \ref{l:image} says
\begin{equation}\label{e:kerj}
|\ker(j_v)|=p^{\epsilon_v}\cdot q_v^{n_v}.
\end{equation}

\begin{lemma}\label{l:sscond} Let $v$ be a place of $K$ and $w$ a place of $k$ sitting over $v$. The group $\ker(j_w)$ consists of all
$\xi$ with $k_{w,\xi}/k_w$ having conductor at most $ \frac{pn_w}{p-1}$.
\end{lemma}
\begin{proof}
An element $\xi\in \ker(j_w)$ can be written as $\partial x$ for some $x\in A(k_w)$. Since $\bar A(\F_w)$ has order prime to $p$, we may choose $x=\iota(b)\in A_o(k_w)$, for some $b\in\mathfrak m_w$.  
Let $a_0,...,a_{p-1}$ be solutions to $\mathscr P(t)=b$.  
Then all $a_u$ are integral over $\O_w$ and
$k_{w,\xi}=k_w(a_0)$. It follows from \eqref{e:diff} that if $\mathrm{Disc}$ is the discriminant of $k_{w,\xi}/k_w$, then
$$\ord_w (\mathrm{Disc})\leq p\cdot n_w.
$$
This implies the conductor of $k_{w,\xi}/k_w$ is at most $\frac{pn_w}{p-1}$. Classes $\xi\in \coh^1(k_w,E_p^{(p)})$ with $k_{w,\xi}/k_w$ unramified are in $\ker(j_w)$ (see \cite[I.3.8]{mil86}). They form a subgroup of order $p$. By the local class field theory,  ramified cyclic extensions of $k_w$
of degree $p$ and conductor at most $m$ are characterized by the group $D_m/D_m^p$, $D_m:=\O_w^*/{1+\pi_w^m\O_w}$. 
In our case $m=\frac{pn_w}{p-1}$ is an integer divisible by $p$ (by \eqref{e:tk}, because each $t_u\in k_w$). 
Since $\O_w=\F_w[[\pi_w]]$, the map
$$D_{\frac{m}{p}}\longrightarrow D_m^p,\;\; x\mapsto x^p,$$
is an bijection. Hence $|D_m/D_m^p|=q_w^{m-1}\cdot (q_w-1)/q_w^{\frac{m}{p}-1}\cdot (q_w-1)=q_w^{n_w}$. In view of \eqref{e:kerj}, the proof is completed by counting.

\end{proof}
The lemma actually says that by \eqref{e:h1hom},
\begin{equation}\label{e:sscond}
\ker(j_w)=\Hom(k_w^*/(1+\pi_w^{\frac{pn_w}{p-1}}\O_w)\cdot (k_w^*)^p, \Z/p\Z).
\end{equation}

\section{Global fields}\label{s:global} Let $\mathcal X/\F_q$ be the complete smooth curve having $K$ as its function field.
Let $\mathcal X_g$, $\X_{go}$ denote the open sets consisting of places where $A$ has good reduction, good ordinary reduction respectively.

\subsection{Poitou-Tate duality}\label{su:ff} 

We first recall the following.

\begin{lemma}\label{l:isogen} Let $\mathrm{f}:B{/K}\longrightarrow B'{/K}$ be a given isogeny of elliptic curves having good reductions at all $v\in \mathcal X_g$ and let
 ${\bf{f}}:\mathcal{B}\longrightarrow \mathcal{B}'$
be the homomorphism extending $\mathrm{f}$ to N$\acute{\text{e}}$ron models over 
$\mathcal{X}_g$. 
Then $\mathcal{N}:=\ker \left[{\bf{f}}\right]$ is a finite flat group scheme 
over $\mathcal{X}_g$.
Furthermore, if $\hat{\mathcal{N}}$ denotes the kernel of the homomorphism $\hat{\bf{f}}: \mathcal{B}'\longrightarrow \mathcal{B}$
extending the dual isogeny $\hat{\mathrm{f}}:B' \longrightarrow B$, then $\mathcal{N}$ 
and $\hat{\mathcal{N}}$ are Cartier dual to each other.
\end{lemma}

Note that the existence and the uniqueness of $\bf f$ and $\hat{\bf f}$ are due to the N$\acute{\text{e}}$ron mapping property, see 
\cite[\S 1.2]{blr90}.

\begin{proof}
Since $\mathcal{B}{/\mathcal{X}_g}$ is an abelian scheme, the morphism 
${\bf{f}}$ is proper. It follows that
$\mathcal{N}{/\mathcal{X}_g}$ is proper and quasi-finite, whence 
finite \cite[8.12.4]{grd67} and flat \cite[III.C.8]{mil86}.
The exact sequence
$$\xymatrix{0\ar[r] & \mathcal{N}\ar[r] & \mathcal{B}
\ar[r]^{\bf{f}} & \mathcal{B}' \ar[r] & 0}$$
together with the isomorphism (see \cite[VIII.7.1]{gth72}, \cite{bbm82} or \cite[III.C.14]{mil86})
$$\mathcal{B}{/\mathcal{X}_g} \simeq 
{\scExt}\;{}^1_{\mathcal{X}_g}( \mathcal{B}, \mathbb{G}_m)$$
induce the exact sequence
$$\xymatrix{{\cdots}\ar[r] & {\scHom}_{_{\mathcal{X}_g}}( \mathcal{B}, \mathbb{G}_m)\ar[r]  & 
{\scHom}_{_{\mathcal{X}_g}}(\mathcal{N}, \mathbb{G}_m) \ar[r]
&  \mathcal{B}' \ar[r]^{\hat{\bf{f}}} 
& \mathcal{B}\ar[r] &{\cdots} . }
$$
Here we apply the commutative diagram
$$\xymatrix{{\scExt}\;{}^1_{\mathcal{X}_g}(\mathcal{B}', \mathbb{G}_m)
 \ar[r]^{{\bf{f}}^*} \ar@{=}[d] & {\scExt}\;{}^1_{\mathcal{X}_g}(\mathcal{B}, \mathbb{G}_m)
  \ar@{=}[d]\\
\mathcal{B}' \ar[r]^{\hat{\bf{f}}} & \mathcal{B}}$$
that extends the already known diagram on the generic fibre.
Then we check the equality ${\scHom}_{_{\mathcal{X}_g}}(\mathcal{B}, \mathbb{G}_m)=0$
fibre-wise by using the fact that over a field every homomorphism from an abelian variety to $\mathbb{G}_m$
is trivial. 

\end{proof}

Let notation be as in Lemma \ref{l:isogen} and let $N$ denote the generic fibre of $\mathcal N$.
For $v\in\mathcal X_g$,
we have (see \cite[\S III.7]{mil86})
\begin{equation}\label{e:h1w}
\xymatrix{\coh^1(\O_v,\mathcal{N})\ar@{^{(}->}[r]  & \coh^1(K_v,N)}.
\end{equation}

\begin{lemma}\label{l:dker} Let notation be as above.
For $v\in\mathcal X_g$, we have
$$\coh^1(\O_v,\mathcal{N})=\ker(\coh^1 (K_v,\mathrm{N}) \longrightarrow  \coh^1(K_v,B)).
$$

\end{lemma}
\begin{proof}
The lemma follows from the fact that $\coh^1 (\O_v,\mathcal{B})=0$ (see \cite[\S III.2.1]{mil86})
and the commutative diagram of exact sequences
$$\xymatrix{ \mathcal{B}(\O_v) \ar[r]\ar@{=}[d] & \mathcal{B}'(\O_v) \ar[r] \ar@{=}[d] & \coh^1 (\O_v,\mathcal{N}) \ar[r]\ar@{^(->}[d] &
\coh^1(\O_v,\mathcal{B}) \ar[d]\\
B(K_v) \ar[r] & B'(K_v) \ar[r]  & \coh^1 (K_v,\mathrm{N}) \ar[r] & \coh^1(K_v,B).}
$$

\end{proof}
Let $\U\subset \X$ be an open subscheme. 
Define
$$ 
\mathcal S(N/\U):= \ker(\coh^1(K,N)\longrightarrow \prod_{v\in \U} \coh^1(K_v,B)).
$$
Denote $\Sel(N/K):=\mathcal S(N/\X)$, it is the kernel of 
$$\xymatrix{\mathcal S(N/\U) \ar[r] & \bigoplus_{v\not\in \U} \coh^1(K_v, B)}.$$
Let $\mathrm Q(N/\U)$ denote the cokernel of the localization map
$$\xymatrixcolsep{5pc}\xymatrix{\mathcal S(N/\U) \ar[r]^-{\mathcal L_{N/K}} & \bigoplus_{v\not\in \U} \coh^1(K_v, N)}.$$

\begin{lemma}\label{l:snk} 
If $\U\subset \X_g$, then $\coh^1(\U,\mathcal N)=\mathcal S(N/\U)$.
\end{lemma}

\begin{proof} Let $\V\subset \U$ be an open affine subscheme.
By the localization sequence \cite[III.0.3(c)]{mil86} and the computation at the beginning of \cite[III.7]{mil86}, 
we have the exact sequence
\begin{equation*}\label{e:lose}
\xymatrix{0\ar[r] & \coh^1(\U,\mathcal{N})\ar[r] & \coh^1(\V,\mathcal{N}) \ar[r]  & \bigoplus_{v\in \U \setminus \V}
\coh^1(K_v,N)/\coh^1(\O_v, \mathcal{N}).}
\end{equation*}
\cite[Lemma 4.2]{gon09} says the natural map
$\coh^1(\V,\mathcal N)\longrightarrow \coh^1(K,N)$ is injective. The exact sequence implies $\coh^1(\U,\mathcal N)\subset\mathcal S(N/\U)$. By \cite[Lemma 2.3]{gon09},
an element in $\mathcal S(N/\U)$ can be obtained from $\coh^1(\V,\mathcal N)$ for some $\V\subset \mathcal U$, and the exact sequence implies it is in $\coh^1(\U,\mathcal{N})$. 

\end{proof}

For $\U\not=\X$, apply the local duality \cite[\S III.6.10]{mil86} and consider the composition 
\begin{equation}\label{e:locdual}
\xymatrixcolsep{5pc}\xymatrix{\mathcal S(C_p/\U) \ar@{^{(}->}[r]^-{\mathcal L_{C_p/\U}} & \bigoplus_{v\not\in \U} \coh^1(K_v, C_p)\ar[r]^-\sim & 
 \prod_{v\not\in \U} \coh^1(K_v, E_p^{(p)})^\vee,}
 \end{equation}
where the injectivity of $\mathcal L_{C_p/\U}$ is due to \cite[Main Theorem]{got12}.

\begin{lemma}\label{l:poitou}
If $\U\subset \X_g$ and $\U\not=\X$, then under {\em{\eqref{e:locdual}}}, the group $\mathcal S(C_p/\U)$ is the Pontryagin dual of $Q(E_p^{(p)}/\U)$.
\end{lemma}
\begin{proof}
Extend $\mathsf F$ and $\mathsf V$ to $\mathcal{F}:\mathcal{A}\longrightarrow \mathcal{A}^{(p)}$ and 
$\mathcal{V}:\mathcal{A}^{(p)}\longrightarrow \mathcal{A}$ over ${\X_g}$.
Denote
$$\mathcal{C}_p =\ker ( \mathcal{F})\;\;\text{and}\;\; \mathcal{E}_{p}^{(p)}=\ker ( \mathcal{V}).$$
They are Cartier dual to each other. By Lemma \ref{l:snk}, we identify $\coh^1(\U,\mathcal C_p)$ with
$\mathcal S(C_p/\U)$.  Then the lemma follows from  Poitou-Tate duality \cite[(5.1.2)]{ces15}.
\end{proof}

In view of Lemma \ref{l:dker}, the following lemme generalizes the fact that for any place $v\in \X_{g}$, the local duality identifies 
$\coh^1(\O_v, \mathcal C_p)\subset \coh^1(K_v,C_p)$ with the annihilator of $\coh^1(\O_w,\mathcal E_p^{(p)})\subset \coh^1(K_v,E_p^{(p)})$ \cite[\S III, Theorem 7.1]{mil86}.

\begin{lemma}\label{l:perp} At each place $v$ of $K$, under the duality $\coh^1(K_v,C_p)=\coh^1(K_v,E_p^{(p)})^\vee$ \cite[\S III. Theorem 6.10]{mil86},
the kernel of $j'_v:\coh^1(K_v,C_p)\longrightarrow  \coh^1(K_v, A)$, as a subgroup of $\coh^1(K_v,C_p)$,
is exactly the annihilator of $\ker(j_v)\subset \coh^1(K_v,E_p^{(p)})$. 
\end{lemma}
We abbreviate the above as 
\begin{equation}\label{e:abb}
\ker (j_v')=\ker(j_v)^\perp.
\end{equation}

\begin{proof} Let $\partial$ denote the connecting homomorphism in the long exact sequence
\begin{equation}\label{e:partial}
\xymatrix{\cdots \ar[r] & A(K_v)\ar[r]^-{\mathsf{F}} & A^{(p)}(K_v) \ar[r]^-{\partial} & \coh^1(K_v,C_p)\ar[r]^-{j'_v} & \coh^1(K_v, A) \ar[r] &\cdots,}
\end{equation}
that, since $p=\mathsf{F}\circ \mathsf{V}$, gives rise to the homomorphism $\bar\partial$ in the diagram
\begin{equation}\label{e:apppartial}
\xymatrix{  A^{(p)}(K_v)/p\cdot A^{(p)}(K_v) \ar[r]^-{\bar\partial} \ar@{^{(}->}[d]^-{i} & \coh^1(K_v,C_p) \ar@{=}[d] \\
 \coh^1(K_v,A_p^{(p)})  \ar[r]^-{\mathsf{V}_*}  & \coh^1(K_v,C_p), }
\end{equation}
where the bottom left-arrow is due to \eqref{e:ve} and  the left down-arrow $i$ is induced from the long exact sequence
$$\xymatrix{  A_p^{(p)}(K_v)\ar[r] &  A^{(p)}(K_v)\ar[r]^-{[p]} &  A^{(p)}(K_v) \ar[r] &  \coh^1(K_v,A_p^{(p)})\ar[r] & \cdots.}
$$
We claim that the diagram \eqref{e:apppartial} is commutative, so that the commutative diagram
$$\xymatrixcolsep{5pc}\xymatrix{  
\coh^1(K_v, A_p^{(p)})\times \coh^1(K_v,A_p^{(p)}) \ar@<-7ex>[d]_-{\mathsf{V}_*} \ar[r]^-{(-,-)_v} & \Q/\Z\ar@{=}[d]\\
\coh^1(K_v,C_p) \times \coh^1(K_v, E_p^{(p)})\ar@<-7ex>[u]_-{j_v} \ar[r]^-{(-,-)_v} & \Q/\Z,}
$$
where  
the left-arrows are local pairings,
yields the commutative diagram (see \cite[\S III, Theorem 7.8]{mil86})
\begin{equation*}\label{e:dualap}
\xymatrixcolsep{5pc}\xymatrix{  
A^{(p)}(K_v)\times \coh^1(K_v,A^{(p)}) \ar@<-7ex>[d]_-{\partial} \ar[r]^-{(-,-)_{A^{(p)}/K_v}} & \Q/\Z\ar@{=}[d]\\
\coh^1(K_v,C_p) \times \coh^1(K_v, E_p^{(p)})\ar@<-7ex>[u]_-{j_v} \ar[r]^-{(-,-)_v} & \Q/\Z.}
\end{equation*}
This shows that $\ker (j_v)^\perp=\image (\partial_v)=\ker(j'_v)$. To prove the claim, we use $\check{\text C}$ech
cocycles (see \cite[\S III.2.10]{mil80}). Let $x\in A^{(p)}(K_v)$ and denote $\bar x$ its image in $ A^{(p)}(K_v)/p\cdot A^{(p)}(K_v)$.
Let $y\in A^{(p)}(K'_w)=\Hom(\mathrm{Spec}( K'_w), A^{(p)})$ be a $p$-division point of $x$ over a finite extension $K'_w$. 
Let $pr_l$, $l=1,2$, be the projection
$$ \mathrm{Spec}( K'_w)\times_{\mathrm{Spec}( K_v)}\mathrm{Spec}( K'_w)\longrightarrow \mathrm{Spec}( K'_w)$$
to the $l$'th factor. Then $\xi:= y\circ pr_1 - y\circ pr_2$ is a $1$-cocycle representing the class $i(\bar x)$. Let $z=\mathsf V(y)
\in \Hom(\mathrm{Spec}( K'_w), A)$ so that $\mathsf F(z)=x$. Then $\mathsf V\circ \xi$ is a $1$-cocycle representing both
$\bar\partial(\bar x)$ and $\mathsf V_*(i(\bar x))$.

\end{proof}

\subsection{The conductor}\label{su:gh1}

Recall that $k=K(E_p^{(p)}(\bar K^s))$ so that $E_p^{(p)}(\bar K^s)=E_p^{(p)}(k)$, on which
the action of the Galois group 
$\Phi:=\Gal(k/K)$ 
is given by 
an injective character 
$$c:\Phi\longrightarrow \F_p^*$$ 
such that $\tensor[^g]x{}=c(g)\cdot x$, for $g\in \Phi$, $x\in E_p^{(p)}(k)$. 
Since the order of $\Phi$ is prime to $p$, the Hochschild-Serre spectral sequence 
(see \cite[\S II.2.21(a)]{mil80}) yields
\begin{equation}\label{e:rest}
\xymatrix{\coh^1(K,E_p^{(p)})\ar[r]^-\sim&  \coh^1(k,E_p^{(p)})^\Phi\ar@{=}[r] & \Hom(G_k, E_p^{(p)}(k))^\Phi}.
\end{equation}

For $w\in S_{ss}(k)$,  let $\iota^{(p)}(t_1)$ be a non-zero element of $E_p^{(p)}(k_w)$ as in \S\ref{su:ss}. Put
 
$$ M_w:=\begin{cases} (1+t_1^p\O_w)\cdot (\O_w^*)^p, & \text{if}\; w\in  S_{ss}(k);\\  
k_w^*, & \text{if}\; w\in \eth(k);\\
\O_w^*, &\text{otherwise}.
\end{cases}
$$ 
Let $\A_k^*$ denote the group of ideles of $k$ and $\mathscr W$ the $p$-completion of
$k^*\backslash \A_k^*/ \prod_{w}M_w$. 

\begin{lemma}\label{l:sele}
We have $\Sel(E_p^{(p)}/k)=\Hom (\mathscr W,E_p^{(p)}(k))$.
\end{lemma}
\begin{proof}
By the global class field theory, $\Hom (\mathscr W,E_p^{(p)}(k))\subset \Hom(G_k, E_p^{(p)}(k))$
consists of elements which are locally trivial at $w\in\eth (k)$, having conductors not greater than
$\ord_w(t_1^p)$ at supersingular places $w$, unramified at others. 
The lemma follows from Lemma \ref{l:mult}, \ref{l:good} and
\ref{l:sscond}.
   
\end{proof}

Every pro-$p$ $\Phi$-module $Y$ can be decomposed as
$Y=\bigoplus_{\chi \in {\hat{\Phi}}} Y^\chi$,
where for each $\chi$, $Y^\chi$ denote the $\chi$-eigenspace $\{y\in Y\;\mid\; \tensor[^g]y{}=\chi(g)\cdot y\}$ .
By \eqref{e:rest} and Lemma \ref{l:sele},
\begin{equation}\label{e:wc}
\Sel(E_p^{(p)}/K)=\Hom(\mathscr W,E_p^{(p)}(k))^\Phi=\Hom(\mathscr W^c,\Z/pZ).
\end{equation}

For $v\in S_{ss}$, put $W_v:=\bigoplus_{w\mid v} k_w^*/((k_w^*)^p\cdot M_w)$, and 
$ U_v:=\bigoplus_{w\mid v} \O_w^*/ M_w$ regarded as a subgroup of $W_v$.

\begin{lemma}\label{l:uc} If $v\in S_{ss}$, then $|U_v^c|=q_v^{n_v}$.
\end{lemma}
\begin{proof}
Again, by the Hochschild-Serre spectral sequence 
\begin{equation}\label{e:restloc}
\xymatrix{\coh^1(K_v,E_p^{(p)})\ar[r]^-\sim&  (\bigoplus_{w\mid v}\coh^1(k_w,E_p^{(p)}))^\Phi
\ar@{=}[r] &(\bigoplus_{w\mid v}\Hom(G_{k_w},E_p^{(p)}(k)))^\Phi .}
\end{equation}
Therefore, Lemma \ref{l:sscond} implies $\ker (j_v)\simeq \Hom(W_v^c, \Z/p\Z)$. Then the lemma
follows from \eqref{e:kerj}, because if $k_w\not=K_v$, then $W_v^c=U_v^c$; if $k_w=K_v$,
then we have the splitting exact sequence
$$0\longrightarrow U_v^c\longrightarrow W_v^c\longrightarrow \Z/p\Z\longrightarrow 0.$$

\end{proof}

\subsubsection{An idelic computation}\label{ss:gen}
In this subsection only, we consider a general situation in which for $w\in S_{ss}(k)$, the $M_w$ in the previous subsection is replaced by
$$M_w:=(1+\pi_w^{a_v}\O_w)\cdot (\O_w^*)^p,$$
where $v$ is a place of $K$ sitting below $w$ and $a_v$ is a chosen integer depending only on $v$,
and we keep $M_w$ unchanged for other $w$. Denote $\mathfrak a:=(a_v)_{v\in S_{ss}}$. Then let $\mathscr W_{\mathfrak a}$ 
be the $p$-completion of $k^*\backslash \A_k^*/ \prod_{w}M_w$ so that $\mathscr W=\mathscr W_{\mathfrak o}$, where 
$\mathfrak o:=(p\cdot \ord_wt_v)_{v\in S_{ss}}$. 
Put
$U_{\mathfrak a}:=\bigoplus_{w\in S_{ss}} \O_w^*/ M_w$,  $W_{\mathfrak a}:=\bigoplus_{w\in S_{ss}}k_w^*/((k_w^*)^p\cdot M_w)$. Assume that
\begin{equation}\label{e:assume}
|U_{\mathfrak a}^c|=\prod_{v\in S_{ss}}q_v^{\rho_v}.
\end{equation}

Let $\bar U^c_{\mathfrak a}$ denote the image of the natural map
$$\varsigma^c_{\mathfrak a}:U_{\mathfrak a}^c \longrightarrow  \mathscr W^c_{\mathfrak a}.$$
Let $V$ be the group of $\eth(k)$-units 
of $k$. Then $\ker(\varsigma^c)$ equals the image of 
$$\varrho^c_{\mathfrak a}:(V/pV)^c\longrightarrow U_{\mathfrak a}^c$$
induced by the localization map $V \longrightarrow \prod_{w\in S_{ss}(k)} \O_w^*$. 
The torsion part of $V$ is finite
of order prime to $p$. The map $V\longrightarrow \prod_{v\in \eth}R_v$, where $R_v:=\prod_{w\mid v} k_w^*/\O_w^*$, 
is injective on the free part of $V$.  The module $R_v$ is fixed by the decomposition subgroup $\Phi_v$ and is the
regular representation of $\Phi/\Phi_v$. 
If $k=K$, the $\Z_p$-rank of $(\Z_p\otimes_{\Z} V)^c$ is $\mathrm{max}\{|\eth_0|-1,0\}$; otherwise it is $|\eth_0|$, so
\begin{equation}\label{e:ethkeric}
|\ker (\varsigma^c_{\mathfrak a})|=|\image \varrho^c_{\mathfrak a}|\leq p^{|\eth_0|}.
\end{equation}
If $\eth=\emptyset$, 
let $\gimel$ denote the $p$-completion of the divisor class group of $k$; otherwise, let $\gimel$ be the $\eth(k)$-class group.
Then $\hslash$ is the $p$-rank of $\gimel$.
The exact sequence
$$\xymatrix{0\ar[r] & \bar U^c_{\mathfrak a}\ar[r] & \mathscr W^c_{\mathfrak a} \ar[r] & \gimel^c \ar[r] & 0} $$
yields
\begin{equation}\label{e:uw}
\xymatrix{0\ar[r] & \bar U^c_{\mathfrak a}\ar[r] & \mathscr W^c_{\mathfrak a}[p] \ar[r]^-{\varkappa^c} & \gimel^c[p] }
\end{equation}
Put
\begin{equation}\label{e:tau}
\tau:=\begin{cases} 
\log_p |\image (\varkappa^c)|+1, & \text{if}\; \eth=\emptyset\;\text{and}\;k=K;\\ 
\log_p |\image (\varkappa^c)|,& \text{otherwise}.
\end{cases}
\end{equation}

Via \eqref{e:rest}, we identify $\coh^1(K,E_p^{(p)})$ with $\Hom(G_k, E_p^{(p)}(k))^\Phi$, so that the conductor of its element 
at each place $v$ is defined. Let $\X_o\subset \X$ be the complement of $S_{ss}$.

\begin{definition}\label{d:asela}
Define $\Sel_{\mathfrak a}(E_p^{(p)}/K)$ to be the subgroup of $\mathcal S(E_p^{(p)}/\X_o)$ consisting of elements having conductors not greater than
$a_v$ at each $v\in S_{ss}$. 
\end{definition}

\begin{lemma}\label{l:av} Assuming \eqref{e:assume}, we have  
$$\log_p|\Sel_{\mathfrak a}(E_p^{(p)}/K)|=\sum_{v\in S_{ss}} \rho_v\cdot \deg v+\tilde\varepsilon_{1}-\tilde\varepsilon_{2},
$$ 
where $\tilde\varepsilon_{1}=\tau\leq \hslash+1$,  $ \tilde\varepsilon_{2}=\log_p|\ker (\varsigma^c_{\mathfrak a})|\leq |\eth_0|\leq  |S_b|$.
\end{lemma}

\begin{proof} Similar to \eqref{e:wc},  $\Sel_{\mathfrak a}(E_p^{(p)}/K)=\Hom (\mathscr W^c_{\mathfrak a},\Z/p\Z)$.
We shall write $\mathscr W_{\mathfrak a}$ additively. 
If $\eth\not=\emptyset$ or $k\not=K$, then $\mathscr W^c_{\mathfrak a}$ is finite with 
$|\mathscr W^c_{\mathfrak a}/p\mathscr W^c_{\mathfrak a}|=|\mathscr W^c_{\mathfrak a}[p]|$; otherwise, $\mathscr W^c_{\mathfrak a}$ is finitely generated over $\Z_p$ of rank $1$, hence
$|\mathscr W^c_{\mathfrak a}/p\mathscr W^c_{\mathfrak a}|=p\cdot |\mathscr W^c_{\mathfrak a}[p]|$.
The lemma follows from \eqref{e:assume}, \eqref{e:ethkeric}, \eqref{e:uw} and \eqref{e:tau}.
\end{proof}

\begin{myproposition}\label{p:qvnv} We have 
$$
\log_p|\Sel(E_p^{(p)}/K)|=\frac{(p-1)\deg \Delta_{A/K}}{12} \cdot\log_p q    +\tilde\varepsilon_{1}-\tilde\varepsilon_{2},
$$ 
where $\tilde\varepsilon_{1}=\tau\leq \hslash+1$,  $ \tilde\varepsilon_{2}=\log_p|\ker (\varsigma^c_{\mathfrak o})|\leq |\eth_0|\leq  |S_b|$.  

\end{myproposition}
\begin{proof} 

In view of \eqref{e:wc},  Lemma \ref{l:uc}, and Lemma \ref{l:av},
we need to show that
\begin{equation}\label{e:discqvnv}
\sum_{v\in S_{ss}} n_v\cdot \deg v=\frac{(p-1)\deg \Delta_{A/K}}{12}.
\end{equation}

Let $\omega$ be an invariant differential of $A/K$ and for each place $v$ let $\omega_{0,v}$ and $\omega_{0,v}^{(p)}$ be
respectively local N$\acute{\text{e}}$ron differentials of $A$ and $A^{(p)}$.  By \cite[\S IV. Corollary 4.3]{sil86},
\begin{equation*}\label{e:sil86}
\ord_v (\mathrm{V}^*\omega_{0,v}/\omega^{(p)}_{0,v})=\ord_v \frac{d\mathscr{P}(t)}{dt}\mid_{_t=0},
\end{equation*}
which together with Lemma \ref{l:ss} and \eqref{e:ord} yield
\begin{equation}\label{e:sumnv}
\sum_{v\in S_{ss}}n_v\cdot\deg v=\sum_{\text{all}\; v} \ord_v (\mathrm{V}^*\omega_{0,v}/\omega^{(p)}_{0,v})\cdot\deg v.
\end{equation}
Now the formula (8) in \cite{tan95} implies
$$\frac{\deg \Delta_{A/K}}{12}=\sum_{\text{all}\;v} \ord_v (\frac{\omega}{\omega_{0,v}})\cdot \deg v=\sum_{\text{all}\;v} \ord_v (V^*(\frac{\omega}{\omega_{0,v}}))\cdot \deg v,
$$
and
$$\frac{p\cdot \deg \Delta_{A/K}}{12}=\frac{\deg \Delta_{A^{(p)}/K}}{12}=\sum_{\text{all}\;v} \ord_v (\frac{V^*\omega}{\omega^{(p)}_{0,v}})\cdot \deg v.
$$
These and \eqref{e:sumnv} lead to the desired equality.

\end{proof}

\subsubsection{Computing $\mathcal S(C_p/\U)$}\label{ss:comp}
Next, we investigate $\mathcal S(C_p/\U)$ for $\U=\X_{go}$, $\X_g$, or $\X$. Let $S_{ngo}$ be the complement of $\X_{go}$ in $\X$, $S'_{ngo}:=\U\cap S_{ngo}$, $S'_{ss}:=\U\cap S_{ss}$, and
$\dag$ the complement of $\U$ in $\X$. Write $\ddag$ for $\dag\cup \eth$.
We first treat the case in which $\X_{go}\not=\X$, or equivalently, $A/K$ is not {\em{isotrivial}}\footnote{Recall that $A/K$ is assumed to be
ordinary, having semi-stable reduction everywhere.}.
Put 
$$\bar W:=\prod_{w\in S_{ngo}(k)} k_w^*/(k_w^*)^p$$
and 
$$\bar{\mathscr W}:=k^*\backslash \A_k^* /(\prod_{w\in \X_{go}(k)} \O_w^*\times \prod_{w\in S_{ngo}(k)} (k^*)^p).$$
If $\ell_c$ and $\ell_{c^{-1}}$ denote the $c$ and $c^{-1}$ eigenspaces of the regular representation of $\Phi$ on $\F_p[\Phi]$, then 
$E_p^{(p)}(k)=\ell_c$. Hence
$$\prod_{w\in S_{ngo}(k)} \coh^1(k_w, E_p^{(p)})=\Hom (\bar W, \ell_c)=\Hom (\bar W\otimes_{\F_p} \ell_{c^{-1}},\Z/p\Z)=(\bar W\otimes_{\F_p} \ell_{c^{-1}})^\vee,$$
and
$$\mathcal S(E_p^{(p)}/\X_{go}\times \mathrm{Spec }\, k)=\Hom(\bar{\mathscr W},\ell_c)=((\bar{\mathscr W}^c/p \bar{\mathscr W})\otimes_{\F_p} \ell_{c^{-1}})^\vee.$$

In view of \eqref{e:locdual} and Lemma \ref{l:poitou},
$$\mathcal S(C_p/\X_{go}\times \mathrm{Spec }\, k)=\ker (\bar W\longrightarrow \bar{\mathscr W}/p\bar{\mathscr W}).$$
Therefore, by the Hochschild-Serre spectral sequence again,
$$\mathcal S(C_p/\X_{go})=\mathcal S(C_p/\X_{go}\times \mathrm{Spec }\, k)^\Phi=\ker (\bar W^c\longrightarrow \bar{\mathscr W}^c/p\bar{\mathscr W}^c).$$

For each $w$, put
$$\overline M_w:=M_w/M_w\cap (k_w^*)^p=M_w\cdot (k_w^*)^p/(k_w^*)^p
.$$  
Lemma \ref{l:perp} gives rise to the left block of the following commutative diagram
$$
\xymatrix{\coh^1(k_w,C_p) \ar[r]^-\sim & \coh^1(k_w,E_p^{(p)})^\vee \ar[r]^-\sim & \Hom(k_w^*/(k_w^*)^p,\Z/p\Z)^\vee 
\ar[r]^-\sim  & k_w^*/(k_w^*)^p\\
\ker (j_w') \ar[r]^-\sim \ar@{^(->}[u] & \ker(j_w)^\perp \ar[r]^-\sim \ar@{^(->}[u] & \Hom(k_w^*/M_w\cdot (k_w^*)^p,\Z/p\Z)^\perp 
\ar[r]^-\sim \ar@{^(->}[u]  & \overline M_w.\ar@{^(->}[u]}
$$
The middle block is obtained by taking the dual of the commutative diagram
\begin{equation}\label{e:homcond}
\xymatrix{\coh^1(k_w, E_p^{(p)}) \ar[r]^-\sim  &   \Hom(k_w^*/(k_w^*)^p, \Z_p/p\Z_p)\\  
\ker (j_w) \ar@{^{(}->}[u]  \ar[r]^-\sim &  
\Hom(k_w^*/(k_w^*)^p\cdot M_w, \Z_p/p\Z_p), \ar@{^{(}->}[u]}
\end{equation} 
which is due to Lemma \ref{l:sscond} together with the local class field theory (see \eqref{e:h1hom} and
\eqref{e:sscond})
while the right block is a direct consequence of duality.

shows that if $\bar {\mathscr M}=\prod_{w\in \dag(k)} k_w^*/(k_w^*)^p\times \prod_{w\in S'_{ngo}(k)} \overline M_w$, then 
\begin{equation}\label{e:scp}
\mathcal S(C_p/\U)=\ker (\bar {\mathscr M}^c\longrightarrow \bar{\mathscr W}^c/p\bar{\mathscr W}^c).
\end{equation}

To proceed further, we introduce 
$$\tilde {\mathscr M}=\prod_{w\in \dag(k)} k_w^*/(\O_p^*)^p \times \prod_{w\in S'_{ngo}(k)} M_w\cdot (\O_p^*)^p/(\O_p^*)^p$$
and 
$$\tilde{\mathscr W};=k^*\backslash \A_k^* /(\prod_{w\in \X_{go}(k)} \O_w^*\times \prod_{w\in S_{ngo}(k)} (\O_w^*)^p).$$
Then $\bar{\mathscr M}=\tilde{\mathscr M}/p\tilde{\mathscr M}$, and since the kernel of 
$\xymatrix{\tilde{\mathscr W}\ar@{->>}[r] & \bar{\mathscr W}}$ is inside $p\tilde{\mathscr W}$, we also have $\bar{\mathscr W}/p\bar{\mathscr W}=\tilde{\mathscr W}/p\tilde{\mathscr W}$. Hence \eqref{e:scp} implies
\begin{equation}\label{e:scp2}
\mathcal S(C_p/\U)=\ker (\tilde {\mathscr M}^c/p\tilde {\mathscr M}^c\longrightarrow \tilde{\mathscr W}^c/p\tilde{\mathscr W}^c).
\end{equation}

Let $\mathscr V$ and $\mathscr M$ denote the kernel and image
of the natural map $\tilde\varsigma:\tilde {\mathscr M} \longrightarrow \tilde{\mathscr W}$.
Let $\tilde \xi\in \mathscr V$ and  let $\xi$ be a lift of $\tilde \xi$ to  $\prod_{w\in \dag(k)} k_w^* \times \prod_{w\in S'_{ngo}(k)} M_w\cdot (\O_p^*)^p$. 
Then
there are $\alpha\in k^*$ and $\theta\in \prod_{w\in \X_{go}(k)} \O_w^*\times \prod_{w\in S_{ngo}(k)} (\O_w^*)^p$
such that
\begin{equation}\label{e:xialphatheta}
\xi=\alpha\cdot \theta.
\end{equation}
Let $V_{\ddag}$ denote the group of $\ddag(k)$-units of $k$. The equality \eqref{e:xialphatheta} implies that $\alpha$ is in the subgroup $V'_{\ddag}\subset V_\ddag$ consisting of those elements which are inside $M_w\cdot (\O_w^*)^p$, for all $w\in S'_{ss}(k)$.
Suppose there is another expression $\xi=\alpha'\cdot\theta'$. Then $\alpha'\alpha^{-1}$ actually belongs to the group $\F_k^*$ of global units.
Since $S_{ngo}=\dag \sqcup S'_{ngo}$, the correspondence $\tilde\xi\leftrightarrow \alpha\pmod{\F_k^*}$ gives rise to an isomorphism
$$\mathscr V \simeq V_{\ddag}'/\F_k^*.$$ 
The exact sequence $\xymatrix{0\ar[r] & \mathscr V \ar[r] & \tilde{\mathscr M}\ar[r] & \mathscr M\ar[r] & 0}$ induces the exact sequence
$$\xymatrix{\tilde{\mathscr M}[p] \ar[r] &\mathscr M[p]\ar[r]^-\partial & \mathscr V/p \mathscr V \ar[r] & \tilde{\mathscr M}/p \tilde{\mathscr M} \ar[r] & \mathscr M/p\mathscr M \ar[r] & 0}.$$

For an $h\in \mathscr M[p]$, let $\tilde \eta$ be one of its preimage in $\tilde{\mathscr M}$ and let 
$\eta$ be a lift of $\tilde \eta$ to $\prod_{w\in \dag(k)} k_w^* \times \prod_{w\in S'_{ngo}(k)} M_w\cdot (\O_p^*)^p$. 
Put $\xi=\eta^p$, so that \eqref{e:xialphatheta} holds for some $\alpha$ and $\theta$, and $\partial(h)$ is represented by $\alpha$. 
In this case, since $\xi$ is a $p$'th power, $\alpha\in (k_w^*)^p$ at all $w\in S_{ngo}(k)$, which is non-empty,
so by the local Leopoldt's conjecture \cite{kis93}, $\alpha=\beta^p$, for some $\beta\in V_\ddag$. Since $pV_\ddag\subset V'_\ddag$, we conclude that
\begin{equation}\label{e:tildemm}
\ker( \tilde{\mathscr M}^c/p \tilde{\mathscr M}^c\longrightarrow  \mathscr M^c/p\mathscr M^c)=(V'_\ddag/pV_\ddag)^c.
\end{equation}

Denote 
$$\beth:=\tilde{\mathscr W}/\mathscr M=k^*\backslash \A_k^* /(\prod_{w\in \X_{go}(k)} \O_w^*\times \prod_{w\in \dag(k)} k_w^* \times \prod_{w\in S'_{ngo}(k)} M_w\cdot (\O_p^*)^p.
$$
Then we have the exact sequence
\begin{equation}\label{e:bethmm}
\xymatrix{\tilde{\mathscr W}[p] \ar[r] & \beth[p] \ar[r] & \mathscr M/p\mathscr M \ar[r] & \tilde{\mathscr W}/p\tilde{\mathscr W}}.
\end{equation}
If $y\in \tilde{\mathscr W}[p]$ is represented by an idele $\zeta=(\zeta_w)_w$, then there are $\alpha\in k^*$ and
$\theta\in \prod_{w\in \X_{go}(k)} \O^*_w\cdot \prod_{w\in S_{ngo}(k)} (\O_w^*)^p$  such that
$$\zeta^p= \alpha\cdot \theta.$$
Then at $w\in S_{ngo}(k)$, $\alpha\in (k_w^*)^p$, so by the local Leopoldt's conjecture
again, $\alpha=\beta^p$, $\beta\in k^*$. Since $y$ is also represented by $\zeta\cdot\beta^{-1}$, we have the isomorphism
\begin{equation*}
\xymatrix{\prod_{w\in S_{ngo}(k)}\O_w^*/(\O_w^*)^p \ar[r]^-{\sim} & \tilde{\mathscr W}[p].}
\end{equation*}
Since the $\gimel$ equals the cokernel of $\prod_{w\in S_{ngo}(k)}\O_w^*/(\O_w^*)^p\longrightarrow \beth$, the above isomorphism and \eqref{e:bethmm} together imply
\begin{equation}\label{e:gimelker}
\ker (\mathscr M^c/p\mathscr M^c \longrightarrow \tilde{\mathscr W}^c/p\tilde{\mathscr W}^c)=\image (\beth[p]^c\longrightarrow \gimel[p]^c)\subset \gimel]p]^c
\end{equation}

We have $\hslash=\gimel[p]$ and
$$\log_p|({V'_\ddag}/p\cdot V_\ddag)^c|\leq \log_p|\Z_p\otimes_\Z V_\ddag^c/p\cdot \Z_p\otimes_\Z V_\ddag^c|=\begin{cases}
0, & \;\text{if}\; \ddag=\emptyset;\\
|\ddag_0-1|, & \;\text{if}\; \ddag\not=\emptyset\;\text{and} \;k=K;\\
|\ddag_0|, & \;\text{otherwise},
\end{cases}$$
so by \eqref{e:tildemm} and \eqref{e:gimelker},
\begin{equation}\label{e:su}
\log_p|\mathcal S(C_p/\U)|\leq \log_p|({V'_\ddag}/p\cdot V_\ddag)^c  |+\log_p|\gimel^c [p] |\leq |\ddag_0|+\hslash.
\end{equation}

Suppose $A/K$ is isotrivial. Then $S_b=S_{ss}=\emptyset$. 
If $k\not=K$, choose a place $v$ not spitting completely over $k/K$; otherwise, choose any $v$.
Set $\natural:=\{v\}$. Let $\mathsf U$ be the complement of $\natural$. Put 
$$\bar W:=\prod_{w\in \natural (k)} k_w^*/(k^*)^p$$
and denote 
$$\bar{\mathscr W}:=k^*\backslash \A_k^* /(\prod_{w\in \mathsf U (k)} \O^*_w\cdot \prod_{w \in \natural (k)} (k^*)^p),$$
so that $ \coh^1(K_v, E_p^{(p)})=\Hom (\bar W^c,\Z/p\Z)$ and
$\mathcal S(E_p^{(p)}/\mathsf U)=\Hom(\bar{\mathscr W}^c/p \bar{\mathscr W}^c,\Z/p\Z)$. 
Thus, in view of \eqref{e:locdual} and Lemma \ref{l:poitou}, if $\bar{\mathscr M}:=\prod_{w\in\natural (k)}\O_w^*/(\O_w^*)^p$, then
$$\mathcal S(C_p/\X)=\ker (\bar{\mathscr M}^c\longrightarrow \bar{\mathscr W}^c/p \bar{\mathscr W}^c).$$

The kernel of $\bar{\mathscr M} \longrightarrow \bar{\mathscr W}$ is in the image of $V_\natural$, the $\natural(k)$-units of $k$. 
Because of our choice, $(\Z_p\otimes_{\Z} V_\natural)^c=0$, and hence
$\bar{\mathscr M}^c \longrightarrow \bar{\mathscr W}^c$ is injective.

Similar to the previous case, the local Leopoldt's conjecture at $w\in\natural(k)$ implies 
$$\xymatrix{\prod_{w\in \natural(k)}k_w^*/(k_w^*)^p \ar@{->>}[r] & \bar{\mathscr W}[p].}$$
Since $(\prod_{w\in \natural(k)}k_w^*/(k_w^*)^p)^c=0$, we have $ \bar{\mathscr W}[p]=0$. 
Now, $(\bar{\mathscr W}/ \bar{\mathscr M})[p]=\gimel[p]$, where $\gimel$ is the divisor class group of $k$.
It follows from the exact sequence

$$\xymatrix{\bar{\mathscr W}^c [p] \ar[r] & \gimel^c [p] \ar[r] &  \bar{\mathscr M}^c  \ar[r] &
\bar{\mathscr W}^c /p\bar{\mathscr W}^c }$$
that
$$|\mathcal S(C_p/\X)|=|\gimel^c [p] |.$$

Thus, the following proposition is proved.
\begin{myproposition}\label{p:selcp} For $\U=\X_{go}$, $\X_g$, or $\X$,
$$\log_p |\mathcal S(C_p/\U)|\leq |\ddag_0|+\hslash.  $$
In particular,
$$\log_p |\Sel(C_p/K)|\leq |\eth_0|+\hslash.$$ 
\end{myproposition}

\begin{proof}[Proof of Proposition \ref{p:sel}] The proposition is a consequence of Proposition \ref{p:qvnv} and
Proposition \ref{p:selcp}, because we have the commutative diagram of exact sequences
$$\xymatrix{ 0 \ar[r] \ar@{=}[d]& \Sel(E_p^{(p)}/K) \ar[r] \ar@{^{(}->}[d] & \Sel_p(A^{(p)}/K) \ar[r] \ar@{^{(}->}[d] & 
\Sel(C_p/K) \ar@{^{(}->}[d] \\
C_p(K) \ar[r]  & \coh^1(K,E_p^{(p)}) \ar[r] \ar[d]& \coh^1(K,A_p^{(p)}) \ar[r]^-{\mathsf V_*} \ar[d] & \coh^1(K,C_p)\ar[d]  \\
 &  \bigoplus_v\coh^1(K_v ,A^{(p)}) \ar@{=}[r]& \bigoplus_v\coh^1(K_v ,A^{(p)}) \ar[r]^-{\mathsf V_*} & \bigoplus_v \coh^1(K_v,A),}$$ 
where the middle long exact sequence is induced from \eqref{e:ve}.
\end{proof}

\subsection{The Cassels-Tate duality}\label{su:p2} The Cassels-Tate pairing induces the perfect pairing (see \cite[III.9.5]{mil86})
$$\langle - , - \rangle_{A/K}:\overline\Sha(A/K)\times \overline\Sha(A/K)\longrightarrow \Q_p/\Z_p.$$
If $\varphi:A\longrightarrow B$ is an isogeny with dual isogeny $\varphi^t$, then the commutative diagram

$$\xymatrixcolsep{5pc}\xymatrix{\overline\Sha(A/K)\times \overline\Sha(A/K) \ar@<-5ex>[d]_{\varphi_{_\natural}} \ar[r]^-{\langle - , - \rangle_{A/K}} & \Q_p/\Z_p\ar@{=}[d]\\
\overline\Sha(B/K)\times  \overline\Sha(B/K) \ar@<-5ex>[u]_{\varphi^t_{_\natural}} \ar[r]^-{\langle - , - \rangle_{B/K}} & \Q_p/\Z_p}
$$
yields the duality between $\overline\Sha(A/K)/\ker (\varphi_{_\natural})$ and $\overline\Sha(B/K)/\ker (\varphi_{_\natural}^t)$. In particular
\begin{equation}\label{e:dualsha}
|(\overline\Sha(A/K)/\ker (\varphi_{_\natural}))[p^\nu]|=
|(\overline\Sha(B/K)/\ker (\varphi^t_{_\natural}))[p^\nu]|,
\end{equation}
since $|G/p^\nu G|=|G[p^\nu]|$ for a finite abelian group $G$. 

\begin{proof}[Proof of Proposition \ref{p:aap}] 
Consider the exact sequence
$$\xymatrix{0\ar[r] & \ker(\mathsf{F}_{_\natural}) \ar[r] & \overline\Sha(A/K)[p^\nu] \ar[r] & 
(\overline\Sha(A/K)/\ker(\mathsf{F}_{_\natural}))[p^\nu] \ar[ld]_{\cdot p^\nu} \\
&  & \ker(\mathsf{F}_{_\natural})\cap p^\nu \overline\Sha(A/K) \ar[r]  & 0,}
$$
where the morphism $\cdot p^\nu$ is induced from the multiplication by $p^\nu$. This implies
\begin{equation}\label{e:shacp}
\log_p|\overline\Sha(A/K)[p^\nu]|=\log_p |(\overline\Sha(A/K)/\ker (\mathsf F_{_\natural}))[p^\nu]|+\delta_1,
\end{equation}
with $ \log_p |\Sel(C_p/K)|\geq  \log_p| \ker(\mathsf{F}_{_\natural}) |\geq \delta_1\geq 0$. Also, consider the exact sequence
$$\xymatrix{0\ar[r] & \ker(\mathsf{V}_{_\natural}) \ar[r] & (\overline\Sha(A^{(p)}/K)/\ker(\mathsf{V}_{_\natural}))[p^\nu] \ar[ld] & \\ 
 & (\overline\Sha(A^{(p)}/K)/\overline\Sha(A^{(p)}/K)[p])[p^\nu] \ar[r]^-{\cdot p^\nu} 
 & \mathsf T \ar[r] & 0,}$$
where $\mathsf T=(\overline\Sha(A^{(p)}/K)[p]/\ker(\mathsf{V}_{_\natural})) \cap p^\nu (\overline\Sha(A^{(p)}/K)/\ker(\mathsf{V}_{_\natural}))
$.
This together with \eqref{e:dualsha} and  \eqref{e:shacp} lead to
\begin{equation}\label{e:shacpp}
\log_p|\overline\Sha(A/K)[p^\nu]|=
\log_p |(\overline\Sha(A^{(p)}/K)/\overline\Sha(A^{(p)}/K)[p])[p^\nu]|+\delta_1+\delta_2,
\end{equation}
with 
$$ \log_p |\Sel(C_p/K)|\geq  \log_p| \ker(\mathsf{F}_{_\natural}) |\geq \log_p|\mathsf{V}_{_\natural}( \overline\Sha(A^{(p)}/K)[p]) |\geq\delta_2\geq 0.$$
Now 
$$(\overline\Sha(A^{(p)}/K)/\overline\Sha(A^{(p)}/K)[p])[p^\nu]=\overline\Sha(A^{(p)}/K)[p^{\nu+1}]/\overline\Sha(A^{(p)}/K)[p].$$
Recall that $r$ denote the co-rank of 
$\Sel_{div}(A/K)\simeq \Sel_{div}(A^{(p)}/K),$
so 
$$\log_p|\overline\Sha(A/K)[p^\nu]|=\log_p|\Sel_{p^\nu}(A/K)|-r\nu,$$
and a similar formula for $ A^{(p)}$.
Therefore,
$$\log_p|\Sel_{p^\nu}(A/K)|=\log_p|\Sel_{p^{\nu+1}}(A^{(p)}/K)|-\log_p|\Sel_{p}(A^{(p)}/K)|+\delta_1+\delta_2.$$
Then we apply Proposition \ref{p:sel} and Proposition \ref{p:selcp}.

\end{proof}




\subsection{The method of specialization}\label{su:spec} 
In this section only, we assume that $A/K$ an {\em{ordinary abelian variety}} defined over a {\em{global field}}. 
As before, let $L/K$ be a
$\Z_p^d$-extension unramified outside a finite set of places of $K$. Let $\Gamma$ be the Galois group, $\Lambda_\Gamma$
the Iwasawa algebra. 

Endow $\Bmu_{p^\infty}$ with the discrete topology and
write $\hat \Gamma$ for the group of all continuous characters $\Gamma\longrightarrow \Bmu_{p^\infty}$. Let
$\O$ be the ring of integers of $\Q_p(\Bmu_{p^\infty})$. 
Every character $\chi\in\hat\Gamma$ extends to a unique continuous $\Z_p$-algebra homomorphism
$\chi:\Lambda_\Gamma\longrightarrow \O$.

For $g\in\Lambda_\Gamma$, not divisible by $p$, by Monsky \cite[Theorem 2.3]{monsky}, the set 
$$\nabla_g:=\{\chi\in\hat\Gamma\;\mid\; \chi(g)\in p\cdot\O\}$$ 
is either $\emptyset$
or is contained in $T_1\cup \cdots\cup T_n$, where for each $i$, there are $\zeta_{i,1},...,\zeta_{i,\nu_i}\in\Bmu_{p^\infty}$, $\nu_i>0$, 
and
$\sigma_{i,1},...,\sigma_{i,\nu_i}\in\Gamma$, extendable to a $\Z_p$-basis of $\Gamma$, such that
$$T_i:=\{\chi\in\hat\Gamma\;\mid\; \chi(\sigma_{i,j})=\zeta_{i,j},\;\text{for}\;j=1,...,\nu_i\}.$$ 
For an element $\psi\in\Gamma$, extendable to a $\Z_p$-basis of $\Gamma$, set
$$T_\psi:=\{\chi\in\hat\Gamma\;\mid\; \chi(\psi)=1\}.$$
Then $T_\psi\subset T_i$ can not hold, unless $\nu_i=1$, $\zeta_{i,1}=1$, and $\sigma_{i,1}$, $\psi$ topologically generate 
the same closed subgroup of $\Gamma$, in such case, we actually have $T_\psi=T_i$. Therefore, there are finitely many rank one 
$\Z_p$-submodules of $\Gamma$ such that if $\psi$ is chosen away from them, then
\begin{equation}\label{e:nabla}
T_\psi\subsetneq \nabla_g.
\end{equation}

For a $\Z_p^e$-subextension $L'/K$ of $L/K$. Denote 
$\Gamma'=\Gal(L'/K)$. The quotient map $\Gamma\longrightarrow\Gamma'  $ extends uniquely to a continuous $\Z_p$-algebra
homomorphism
$$p_{L/L'}:\Lambda_\Gamma\longrightarrow\Lambda_{\Gamma'}.$$
Write $\Psi:=\Gal(L/L')$ so that $\Gamma'=\Gamma/\Psi$. Put $\mathscr I_\Psi:=\ker(p_{L/L'})$.



Let $\mathfrak b$ be a finite set of places of $K$. For a subextension $M$ of $L/K$, let
$\mathfrak b_{M}\subset \mathfrak b$ denote the subset consisting of places splitting completely over $M$.

\begin{lemma}\label{l:monsky} Suppose $d\geq 2$. Let $Z_\ell$, $\ell=1,..,s$, be finitely generated torsion $\Lambda_\Gamma$-modules, 
$\theta$ an element in $\Lambda_\Gamma$, not divisible by $p$, and $\mathfrak b$ a finite set of places of $K$.
There is an element $\psi\in\Gamma$ extendable to a $\Z_p$-basis such that if $L'$ is the fixed field of $\psi$, then
$\mathfrak b_{L'}=\mathfrak b_L$, $p_{L/L'}(\theta)$ is not divisible by $p$,
and each $\Lambda_{\Gamma'}$-module $Z'_\ell:=Z_\ell/\mathscr I_\Psi Z_\ell$ is torsion, having the same elementary $\mu$-invariants 
 as those of $Z_\ell$ over $\Lambda_\Gamma$.
\end{lemma}
\begin{proof} 
For each $v\in \mathfrak b$ with non-trivial decomposition subgroup $\Gamma_v$, if $\psi\not\in \Gamma_v$ or $\Gamma_v$ is of $\Z_p$-rank greater than one,  then $v$ does not split completely over $L'$. To have $\mathfrak b_{L'}=\mathfrak b_L$, we only need to choose $\psi$ away from all rank one $\Gamma_v$, $v\in\mathfrak b$. 

Similar to \eqref{e:iwasawa}, we have
\begin{equation}\label{e:iwasawal}
\xymatrix{0 \ar[r] &  \bigoplus_{i=1}^{m_\ell} \Lambda_\Gamma/(p^{\alpha_{\ell,i}}) \oplus \bigoplus_{j=1}^{n_\ell} \Lambda_\Gamma/(\eta_{\ell,j}^{\beta_{\ell,j}})
\ar[r]  &  Z_\ell \ar[r] &  N_\ell \ar[r] & 0,}
\end{equation}
where $N_\ell$ is pseudo-null and every $\eta_{\ell,j}$ is not divisible by $p$. Choose  for each $\ell$, an annihilator $h_\ell\in\Lambda_\Gamma$ of $N_\ell$, not divisible by $p$, and put 
$$g:=\theta\cdot \prod_{\ell=1}^s h_\ell\cdot \eta_{\ell,1}\cdot\cdots\cdot \eta_{\ell,n_\ell}.$$
We also choose $\psi$ satisfying \eqref{e:nabla}. Since $\hat{\Gamma}'=T_\psi$, there exists $\chi\in\hat\Gamma'$ such that
$\chi(p_{L/L'}(g))\not\in p\cdot \O$, so $ p_{L/L'}(g)\not\in p\cdot\Lambda_{\Gamma'}$.
Because $p^\alpha\cdot g^\beta\cdot Z_\ell=0$, for some $\alpha,\beta\in\Z$, we have $p^\alpha\cdot p_{L/L'}(g^\beta)\cdot Z'_\ell=0$.
Hence $Z'_\ell$ is torsion over $\Lambda_{\Gamma'}$.
 

To compare the elementary $\mu$-invariants of $Z_\ell$ and $Z'_\ell$, we apply the maps of multiplication by $\psi-1$ to \eqref{e:iwasawal} and use the snake lemma
to obtain the exact sequence
\begin{equation}\label{e:iwasawall}
N_\ell[\psi-1] \longrightarrow \bigoplus_{i=1}^{m_\ell} \Lambda_{\Gamma'}/(p^{\alpha_{\ell,i}}) \oplus \bigoplus_{j=1}^{n_\ell} \Lambda_{\Gamma'}/
(p_{L/L'}(\eta_{\ell,j})^{\beta_{\ell,j}})
\longrightarrow  Z'_\ell \longrightarrow   N_\ell /\mathscr I_\Psi N_\ell.
\end{equation}
Because $N_\ell[\psi-1]$ is annihilated by  $p_{L/L'}(g)$ which is relatively prime to $p$, the second arrow in the exact sequence
is injective on $\bigoplus_{i=1}^{m_\ell} \Lambda_{\Gamma'}/(p^{\alpha_{\ell,i}})$. 
We complete the proof by comparing the $p$th power factors of the characteristic ideals of items in the sequence.


\end{proof}

In Lemma \ref{l:monsky}, the field $L'$ is a $\Z^{d-1}_p$-extension of $K$.
By repeatedly applying the lemma, we obtain sequences
\begin{equation}\label{e:li}
L\supset L'\supset \cdots \supset L^{(d-1)},
\end{equation}
and, for each $\ell$, 
\begin{equation}\label{e:zi}
Z_\ell\longrightarrow Z'_\ell\longrightarrow \cdots \longrightarrow Z_\ell^{(d-1)}.
\end{equation}
Put $\Psi_0=\Psi$, $\Psi_i=\Gal(L/L^{(i+1)})$, and $\Gamma^{(i+1)}=\Gal(L^{(i+1)}/K)=\Gamma/\Psi_i$.
Then $L^{(d-1)}$ is a $\Z_p$-extension of $K$ with $\mathfrak b_{L^{(d-1)}}=\mathfrak b_L$, $p_{L/L^{(d-1)}}(\theta)$ not divisible by $p$,
and for each $\ell$, the $\Lambda_{\Gamma^{(d-1)}}$-module $Z_\ell^{(d-1)}=Z_\ell/\mathscr I_{\Psi_{d-2}} Z_\ell$ has the same elementary $\mu$-invariants as those of $Z_\ell$ over $\Lambda_\Gamma$. These elementary $\mu$-invariants $p^{\alpha_{\ell,1}},...,p^{\alpha_{\ell,m_\ell}}$ can be recovered by using the counting formula below. For each $\nu$, define 
$$\alpha_{\ell, i,\nu}=\mathrm{min}\{\nu, \alpha_{\ell,i}\}.$$
Let $\sigma\in\Gamma^{(d-1)}$ be a topological generator and put $x=\sigma-1$.
Let $\mathscr J_{\nu,n}$ denote the ideal of $\Lambda_{\Gamma^{(d-1)}}$ generated by $p^\nu$ and $(x+1)^{p^n}-1$.

\begin{lemma}\label{l:counting} Let the notation be as above. 
Then
$$\log_p |Z^{(d-1)}_\ell/\mathscr J_{\nu,n} Z^{(d-1)}_\ell|=p^n\cdot \sum_{i=1}^{m_\ell} \alpha_{\ell, i,\nu}+\mathrm{O}(1).$$
\end{lemma}
\begin{proof} Taking $Z=Z^{(d-1)}_\ell/p^\nu Z^{(d-1)}_\ell$ in \eqref{e:iwasawa}, we deduce the following exact sequence, 
\begin{equation}\label{e:zpnuz}
\xymatrix{0 \ar[r] &  \bigoplus_{i=1}^{m_\ell} \Lambda_{\Gamma^{(d-1)}}/(p^{\alpha_{\ell,i,\nu}})\ar[r]  &  Z^{(d-1)}_\ell/p^\nu Z^{(d-1)}_\ell \ar[r] &  N_{\ell,\nu} \ar[r] & 0,}
\end{equation}
where $N_{\ell,\nu}$ is finite, since the other two items are pseudo isomorphic.
Write $R_{\alpha}$ for $\Z_p/p^\alpha\Z_p$. The lemma is a consequence of the exact sequence induced from \eqref{e:zpnuz}:
$$\xymatrix{N_0[\sigma^{p^n}-1]\ar[r] &   \bigoplus_{i=1}^{m_\ell} R_{\alpha_{\ell,i,\nu}}[x]/((x+1)^{p^n}-1) \ar[r]  &  Z_\ell^{(d-1)}/\mathscr J_{\nu,n} Z_\ell^{(d-1)}
\ar[dl] \\
& N_0/(\sigma^{p^n}-1) N_0.}
$$
\end{proof}

To apply the above to dual Selmer groups, we need the following simplified control lemma.
For $K\subset F \subset L$, consider the restriction maps
$$\mathrm{res}_{L/F}^{(\nu)}:\Sel_{p^\nu}(A/F)\longrightarrow \Sel_{p^\nu}(A/L)^{\Gal(L/F)}.   $$

 Let $\texttt K^{(n)}$ denote the $n$th layer of the $\Z_p$-extension $L^{(d-1)}/K$.
 Let $\mathfrak r$ denote the ramification locus of $L/K$, which is assumed to be finite.
 
\begin{lemma}\label{l:control} Let $L^{(d-1)}/K$ be an intermediate $\Z_p$-extension of $L/K$.
Suppose $\mathfrak r$ contains only places where $A$
has good ordinary reduction or multiplicative reduction,
$\mathfrak b_{L^{(d-1)}}=\mathfrak b_L$ and $\mathfrak b$ contains $\mathfrak r$ as well as all places where 
$A$ has bad reduction. 
For a given $\nu$, the orders of $\ker(\mathrm{res}^{(\nu)}_{L/\texttt K^{(n)}})$ and 
$\coker(\mathrm{res}^{(\nu)}_{L/\texttt K^{(n)}})$ are bounded, as $n$ varies.
\end{lemma}
\begin{proof} Let $M=A_{p^\nu}(L)=A_{p^\nu}(K')$ for some finite sub-extension $K'/K$ of $L/K$. Write $\mathtt K'^{(n)}$ for $\mathtt K^{(n)}K'$.
\cite[Lemma 3.2.1]{tan10} says that for $i=0,1,2$,
$$|\coh^i(L/\texttt K'^{(n)}, M)|\leq |M|^{d^i}.$$ 
Since $[\mathtt K'^{(n)}:\mathtt K^{(n)}]\leq [K':K]$, by counting the number of co-chains, we deduce
$$|\coh^i( \mathtt K'^{(n)}/\mathtt K^{(n)},M)|\leq |M|^{[K':K]^i}.$$

This bounds $|\ker(\mathrm{res}^{(\nu)}_{L/\texttt K^{(n)}})|$. To bound $|\coker(\mathrm{res}^{(\nu)}_{L/\texttt K^{(n)}})|$, by the Hochschild-Serre spectral sequence, we need to bound (see the proof of \cite[Theorem 4]{tan10})
$$|\bigoplus_{\text{all}\;v}\prod_{w\mid v} \coh^1(L_w/\texttt K^{(n)}_{w},A(L_w))[p^\nu]|.$$ 

Write $\mathscr H^{(\nu)}_w$ for $\coh^1(L_w/\texttt K^{(n)}_w,A(L_w))[p^\nu]$.
If $v\not\in \mathfrak b$,  then $\mathscr H^{(\nu)}_w=0$ \cite[I.3.8]{mil86}, for all $w\mid v$. Also, $\mathscr H^{(\nu)}_w=0$,
for all $w$ sitting over $\mathfrak b_{L}$, because $\texttt K^{(n)}_w=L_w$. 

Suppose $v\in\mathfrak b$ but $v\not\in\mathfrak b_L=\mathfrak b_{L^{(d-1)}}$.
The number of places of $\texttt K^{(n)}$ sitting over $v$ is bounded as $n$ varies. We need to bound the order of $\mathscr H^{(\nu)}_w$,
for all $w\mid v$. If $A$ has good ordinary reduction at $v$, by \cite[(3) and Theorem 2]{tan10}, 
the order of $\mathscr H^{(\nu)}_w$ is bounded by $p^{2\nu(d+1)\dim A}$. It is well-known (for instance, see the last two paragraphs of \cite{tan10})
that if $A$ has split multiplicative reduction at $v$, the order of $\mathscr H^{(\nu)}_w$ is bounded by 
$p^{\nu d\dim A }$. In general, if $A$ has multiplicative reduction at $v$, then over some unramified extension $K'_v/K_v$, the reduction of $A$ becomes split multiplicative. Since $\mathscr H_w^{(\nu)}\subset \coh^1(L_wK'_v/\mathtt K_v, A(L_wK'_v))$, writing $\mathtt K'_v$,
$L'_w$ for $\mathtt K_vK'_v$, $L_wK'_v$, we end the proof by using the exact sequence
$$\xymatrix{\coh^1(\mathtt K'_v/\mathtt K_v, A(\mathtt K'_v)) \ar@{^{(}->}[r] & \coh^1(L'_w/\mathtt K_v, A(L'_w))\ar[r] &
\coh^1(L'_w/\mathtt K'_v ,A(L'_w))}$$
and the fact that the component group of $A/K'_v$ has $p$-rank bounded by $\dim A$ (see \cite[Proposition 5.2]{bx96})
so that by \cite[Proposition I.3.8]{mil86} the order of $\coh^1(\mathtt K'_v/\mathtt K_v, A(\mathtt K'_v))$ is bounded.

\end{proof}

\begin{lemma}\label{l:special} Let $\theta\in\Lambda_\Gamma$ be an element not divisible by $p$.
Let $A_\ell$, $\ell=1,...,s$, be ordinary abelian varieties defined over $K$ such that all
$Z_\ell:=\Sel_{p^\infty}(A_\ell/L)^\vee$ are torsion over $\Lambda_{\Gamma}$ and the ramification locus $\mathfrak r$
contains only places where each $A_\ell$ has either good ordinary reduction or multiplicative reduction.
Let $\mathfrak b$ be a finite set of places of $K$ containing $\mathfrak r$ and all places where some $A_\ell$ has bad reduction.
Assume that $p^{\alpha_{\ell,1}},...,p^{\alpha_{\ell,m_\ell}}$ are elementary $\mu$-invariants of $Z_\ell$.
There exists an intermediate 
$\Z_p$-extension $L^{(d-1)}/K$ of $L/K$ such that the following holds:
\begin{enumerate}
\item[(a)] $p_{L/L^{(d-1)}}(\theta)\not\in p\Lambda_{\Gamma^{(d-1)}}$, where $\Gamma^{(d-1)}=\Gal(L^{(d-1)}/K)$.
\item[(b)] $\mathfrak b_{L^{(d-1)}}=\mathfrak b_L$.
\item[(c)] For each $\ell$, the elementary $\mu$-invariants of $\Sel_{p^\infty}(A_\ell/L^{(d-1)})^\vee$ over $\Lambda_{\Gamma^{(d-1)}}$ are the same as those of $Z_\ell$ over $\Lambda_\Gamma$.
\item[(d)] In particular, if $L/K$ is a $\Z_p$-extension and $K^{(n)}$ denote the $n$th layer, then
\begin{equation}\label{e:anpnu}
\log_p |\Sel_{p^\nu}(A_\ell/K^{(n)})|=p^n\cdot \sum_{i=1}^{m_\ell} \alpha_{\ell, i,\nu}+\mathrm{O}(1).
\end{equation}
\end{enumerate}

\end{lemma}
\begin{proof} (a) and (b) are from Lemma \ref{l:monsky}. Observe that $Z^{(d-1)}_\ell/\mathscr J_{\nu,n} Z^{(d-1)}$ is nothing but
the Pontryagin dual of $\Sel_{p^\nu}(A_\ell/L)^{\Gal(L/\mathtt K^{(n)})}$, so by Lemma \ref{l:counting}, \ref{l:control}, we have
\begin{equation}\label{e:selnu1}
\log_p |\Sel_{p^\nu}(A_\ell/\mathtt K^{(n)})|=p^n\cdot \sum_{i=1}^{m_\ell} \alpha_{\ell, i,\nu}+\mathrm{O}(1).
\end{equation}
In the situation of (d), $L=L^{(d-1)}$, and hence, \eqref{e:anpnu} holds.
To show (c), we assume that the elementary $\mu$-invariants of $\Sel_{p^\infty}(A_\ell/L^{(d-1)})^\vee$ over $\Lambda_{\Gamma^{(d-1)}}$ are
$p^{\alpha'_{\ell,1}},...,p^{\alpha'_{\ell,w_\ell}}$. Apply (d) to $L^{(d-1)}/K$ and obtain
$$\log_p |\Sel_{p^\nu}(A_\ell/\mathtt K^{(n)})|=p^n\cdot \sum_{i=1}^{w_\ell} \alpha'_{\ell, i,\nu}+\mathrm{O}(1),
$$
where, as before, $\alpha'_{\ell, i,\nu}:=\mathrm{min}\{\alpha'_{\ell, i},\nu\}$.
This and \eqref{e:selnu1} leads to
$$\sum_{i=1}^{w_\ell} \alpha'_{\ell, i,\nu}=\sum_{i=1}^{m_\ell} \alpha_{\ell, i,\nu},$$
for all $\nu$. We conclude that $w_\ell=m_\ell$ and $p^{\alpha'_{\ell,1}},...,p^{\alpha'_{\ell,w_\ell}}$ are the same as $p^{\alpha_{\ell,1}},...,p^{\alpha_{\ell,m_\ell}}$.

\end{proof}

\subsection{The elementary $\mu$-invariants}\label{su:last} In this section, we complete the proof of Proposition \ref{p:p} and Proposition \ref{p:gen}. For each $\Z_p$-subextension $F/K$ of $L/K$, write
$$S_b(F)=\eth'_F\sqcup \eth_F,$$
where if $p=2$, $\eth'_F$ is the set of places at which $A/K$ has non-split multiplicative reduction such that
the group of components is of even order; if $p\not=2$, $\eth'_F=\emptyset$. For $w\in S_b(F)$ sitting on $v\in S_b$,
if $w\in \eth'_F$ and $v\in\eth$, or $w\in \eth_F$ and $v\in \eth'$, then $F_w\not=K_v$, and hence there is only finitely many places of $F$ sitting over $v$.




\begin{proof}[Proof of Proposition \ref{p:p}]  We apply Lemma \ref{l:special}, taking $s=1$, $Z_1=X^{(p)}$, 
$\mathfrak b=\mathfrak r\cup S_b$. 
Let $\mathtt K^{(n)}$ be the $n$th layer of $L^{(d-1)}/K$. Since the degrees of $k/K$ and $\mathtt K^{(n)}/K$ are relatively prime,
one see that a place $v$ of $K$ splits completely in $k$, if and only if 
all places of $\mathtt K^{(n)}$ sitting over $v$ split completely in $k\mathtt K^{(n)}$, because both assertions are equivalent to that the decomposition subgroup $\Gal(k\mathtt K^{(n)}/K)_v$ contains no non-trivial element in $\Gal(k\mathtt K^{(n)}/\mathtt K^{(n)})$, and hence
contained in $\Gal(k\mathtt K^{(n)}/k)$. 
Put 
$$\eth_{0,n}:=\{w\in\eth_{\mathtt K^{(n)}}\;\mid\; w\;\text{splits completely over}\; k\mathtt K^{(n)}\}.$$ 
Then, by the discussion at the beginning of this section,
$$|\eth_{0,n}|=|\eth_0(\mathtt K^{(n)})|+\mathrm{O}(1)=p^n\cdot |\eth_1|+\mathrm O(1).$$
If $\F_{q^{(n)}}$ denote the constant field of $\mathtt K^{(n)}$, then
$$\deg \Delta_{A/\mathtt K^{(n)}}\cdot \log_pq^{(n)}=p^n\cdot \deg \Delta_{A/K}\cdot \log_p q.$$
Therefore, Lemma \ref{l:special}(d) and Proposition \ref{p:sel} say if $m$ is the $\mu$-rank of $A^{(p)}/L$, then 
$$p^n\cdot m=\log_p|\Sel_p(A^{(p)}/\mathtt K^{(n)})|+\mathrm{O}(1)\geq p^n\cdot (\frac{(p-1)\deg \Delta_{A/K}}{12}\cdot\log_p q-|\eth_1|)+\mathrm O(1).$$
This proves the proposition.
\end{proof}

\begin{proof}[Proof of Proposition \ref{p:gen}] 
Take $s=2$, $Z_1=X^{(p)}$, $Z_2=X$, $\mathfrak b=S_b\cup \mathfrak r$, $\theta=\Theta_L$, , so that $m=m_1$, and $\alpha_i=\alpha_{1,i}$, for $i=1,...,m$.  By Lemma \ref{l:special}, $\mathfrak b_{L^{(d-1)}}=\mathfrak b_L=\emptyset$ and $\Theta_{L^{(d-1)}}$ is not divisible by $p$. 
Thus, we may assume that $d=1$. 

Let $K^{(n)}$ denote the $n$th layer of $L/K$.
If $L=K_p^{(\infty)}$, we have shown in \S\ref{su:main} that $|\mathfrak w_{K^{(n)}}[p]|=\mathrm O(1)$; otherwise,
for $n$ sufficiently large, $L/K^{(n)}$ is totally ramified at certain place, so that $\Hom(\Gal(L/K^{(n)}),\Q_p/\Z_p)\cap \Hom(\mathfrak w_{K^{(n)}}, \Q_p/\Z_p)=\{0\}$. Hence, $\Hom(\mathfrak w_{K^{(n)}}, \Q_p/\Z_p)\longrightarrow \Hom(\mathfrak w_L, \Q_p/\Z_p)$ is injective, or equivalently, the map $\mathfrak w_L\longrightarrow
\mathfrak w_{K^{(n)}}$ is surjective, for sufficiently large $n$. Since $p\nmid \Theta_L$, $\mathfrak w_L$ has trivial $p$-part, the order of
$\mathfrak w_{K^{(n)}}[p]$ must be bounded. The assumption says
$$|S_b(K^{(n)})|=\mathrm O(1).$$
Therefore, by Lemma \ref{l:special}(d) 
and  Proposition \ref{p:sel}, we obtain
$$\begin{array}{rcl}
p^n \cdot m& =& p^n\cdot \sum_{i=1}^m\alpha_{1,i,1}\\
& =&\log_p |\Sel_p(A^{(p)}/K^{(n)})|+\mathrm O(1)\\
&=& p^n\cdot\frac{ (p-1)\deg\Delta_{A/K}}{12}\cdot \log_p q+\mathrm O(1),
\end{array}$$
that proves the first assertion. Then Lemma \ref{l:special}(d) and Proposition \ref{p:aap} lead to
$$\begin{array}{rcl}
p^n\cdot \sum_{i=1}^{m_2} \alpha_{2, i,\nu}&=&\log_p |\Sel_{p^\nu}(A/K^{(n)})|+\mathrm{O}(1)\\
&=& p^n\cdot(\log_p |\Sel_{p^{\nu+1}}(A^{(p)}/K^{(n)})|- \frac{ (p-1)\deg\Delta_{A/K}}{12}\cdot \log_p q)+\mathrm{O}(1)\\
&=& p^n\cdot \sum_{i=1}^m(\alpha_{1,i,\nu+1}-1)+\mathrm{O}(1),\end{array}
$$
which holds for every $\nu$, so the proposition is proved.

\end{proof}

\end{document}